\documentclass[12pt, a4paper, reqno]{amsart}

\usepackage{amsmath, amssymb, amsthm, mathrsfs}
\usepackage{courier}
\usepackage[T1]{fontenc}
\usepackage[utf8]{inputenc}
\usepackage[all]{xy}
\usepackage[
  margin=1.5cm,
  includefoot,
  footskip=30pt,
  ]{geometry}
\usepackage[colorlinks=true,citecolor=cyan,linkcolor=magenta, urlcolor=blue, filecolor=green
]{hyperref}
\usepackage{hyperref}
\usepackage{indentfirst}

\newcommand{\cF}{\mathcal{F}}

\newcommand{\cO}{\mathcal{O}}

\newcommand{\bA}{\mathbb{A}}
\newcommand{\bC}{\mathbb{C}}
\newcommand{\bF}{\mathbb{F}}

\newcommand{\bQ}{\mathbb{Q}}

\newcommand{\bZ}{\mathbb{Z}}

\newcommand{\fl}{\mathfrak{l}}
\newcommand{\ff}{\mathfrak{f}}
\newcommand{\fm}{\mathfrak{m}}
\newcommand{\fp}{\mathfrak{p}}

\newcommand{\fP}{\mathfrak{P}}

\DeclareMathOperator{\Gal}{Gal}
\DeclareMathOperator{\GL}{GL}

\DeclareMathOperator{\Norm}{Nm}
\DeclareMathOperator{\Nm}{Nm}

\DeclareMathOperator{\Sp}{Sp}

\DeclareMathOperator{\Tr}{Tr}

\DeclareMathOperator{\ch}{char}

\DeclareMathOperator{\meas}{meas}
\DeclareMathOperator{\ord}{ord}

\theoremstyle{plain}% default style
\newtheorem{theorem}{Theorem}[section]
\newtheorem{lemma}[theorem]{Lemma}
\newtheorem{proposition}[theorem]{Proposition}
\newtheorem{corollary}[theorem]{Corollary}
\newtheorem{condition}[theorem]{Condition}

\theoremstyle{definition} % definition style

\newtheorem{remark}[theorem]{Remark}

\theoremstyle{remark} % remark style
\newtheorem*{example}{Example}

\DeclareFontFamily{U}{wncy}{}
\DeclareFontShape{U}{wncy}{m}{n}{<->wncyr10}{}
\DeclareSymbolFont{mcy}{U}{wncy}{m}{n}
\DeclareMathSymbol{\Sha}{\mathord}{mcy}{"58}

\newcommand{\lp}{\left(}
\newcommand{\rp}{\right)}

\newcommand{\lara}[1]{\langle #1 \rangle}
\newcommand{\lbrb}[1]{\lp #1 \rp}
\newcommand{\lcrc}[1]{\left\{ #1 \right\}}

\newcommand{\quadsym}[2]{\left( \frac{#1}{#2} \right)}
\newcommand{\pthsym}[2]{\left( \frac{#1}{#2} \right)_p}

\title{Nonvanishing of $L$-function of some Hecke characters on cyclotomic fields}
\author{Keunyoung Jeong}
\address{Department of Mathematics Education, Chonnam National University, 77, Yongbong-ro, Buk-gu, Gwangju 61186, Korea}
\email{keunyoung@jnu.ac.kr}

\author{Yeong-Wook Kwon}
\address{Department of Mathematical Sciences, Ulsan National Institute of Science and Technology, UNIST-gil 50, Ulsan 44919, Korea}
\email{ywkwon@unist.ac.kr}

\author{Junyeong Park}
\address{Department of Mathematical Sciences, Ulsan National Institute of Science and Technology, UNIST-gil 50, Ulsan 44919, Korea}
\email{junyeongp@gmail.com}


\begin{document}

\maketitle

\begin{abstract}
In this paper, we show the nonvanishing of some Hecke characters on cyclotomic fields.
The main ingredient of this paper is a computation of eigenfunctions and the action of Weil representation at some primes including the primes above $2$.
As an application, we show that 
for each isogeny factor of the Jacobian of the $p$-th Fermat curve where $2$ is a quadratic residue modulo $p$, there are infinitely many twists whose analytic rank is zero.
Also, for a certain hyperelliptic curve over the $11$-th cyclotomic field whose Jacobian has complex multiplication, there are infinitely many twists whose analytic rank is zero.
\end{abstract}

\section{Introduction}

In his work \cite{Yan97}, Yang gave an explicit relation between the special value of the $L$-function of the Hecke character and certain theta liftings.
To compute the theta lifting explicitly one needs to study the action of Weil representation on the space of Schwarz functions, and its eigenfunction at each place.
Yang and Stoll--Yang \cite{Yan97, Yan99, SY03} gave this information for some places.
Using these results, the nonvanishing of $L$-functions of certain elliptic curves \cite{Yan97, Yan99}, hyperelliptic curves \cite{SY03, JPY} was proved. 

In this paper, we show nonvanishing of $L$-functions of some Hecke characters by considering Weil representation and certain choices of eigenfunctions.
Yang \cite{Yan99} gave nice choices of the eigenfunctions over many places. However, a particular concern is needed at primes above $2$.
\cite{SY03} is a nice illustration of this problem for Hecke characters on $\bQ(\zeta_5)$, attached to the hyperelliptic curves $y^2 = x^5 + A$ where $A$ is an integer.
In contrast, we try to consider as many Hecke characters as possible. 
Consequently, we can obtain the following version of the nonvanishing theorem, which covers several Hecke characters over the cyclotomic field.

\begin{theorem} \label{main:main}
	Let $K$ be the $p$-th cyclotomic field, let $F$ be its maximal totally real subfield, and let $\epsilon_{K/F}$ be the nontrivial quadratic character on $\bA_{F}^\times$. Let $\chi$ be a Hecke character on $K$ whose restriction to $F$ is $\epsilon_{K/F}$ and root number is $+1$. 
	We further assume that $\chi$ satisfies the following conditions:
	\begin{itemize}
	\item (Infinite place condition) For each real place $v$ of $F$ and a complex place $w \mid v$, $\chi_w(z) = |z|_w/z$. Also, $\alpha$ in the equation (\ref{eqn:alpha}) of Theorem \ref{thm:Yang97main} is totally positive.
	\item (Splitting condition) If a prime $v \nmid 2$ of $F$ divides the conductor of $\chi$, then $v$ splits in $K/F$.
	\item ($p$-condition) The conductor exponent of $\chi$ at the unique prime above $p$ is $1$.
	\item (Weil index condition) If primes above $2$ are inert in $K/F$, $\chi$ satisfies % (\ref{eqn:Weilcondat1}) and (\ref{eqn:Weilcondatzeta}) in 
        Condition \ref{cond:Weil}.
	\end{itemize}
	Then $L\left(1, \chi\right) \neq 0$.	
\end{theorem}
%For the precise statement of the Weil index condition, see Condition \ref{cond:Weil}. We also note that the definition of $\alpha$ is given in Section \ref{sec:sketch}.

As an application, we get an infinite family of Hecke characters with nonvanishing $L$-function from twists of the Fermat curve.
Given an odd prime number $p$, let
\begin{align*}
    \cF(p) : x^p + y^p = 1
\end{align*}
be the $p$-th Fermat curve, and let $J(p)$ be its Jacobian.
For integers $r, s, t$ satisfying $r + s + t = p$, we define a curve
\begin{align*}
    C_{r, s, t} : y^p = x^r(1-x)^s
\end{align*}
and let $J_{r, s, t}$ be its Jacobian.
It is known that $J(p)$ and $J_{1,1,p-2}\times J_{1,2,p-3}\times\cdots\times J_{1,p-2,1}$ are isogenous (cf. \cite[\S 2]{GR78}).
Furthermore, $J_{r, s, t}$ has complex multiplication. There is a Hecke character $\chi_{r,s,t}$ whose $L$-function is equal to the $L$-function of $J_{r,s,t}$.
%Actually, this is a Hecke character coming from a Jacobi sum.
The properties of $\chi_{r, s, t}$ were extensively studied by \cite{GR78, Roh92}.
More recently, Shu \cite{Shu} studied the twists of $\mathcal{F}(p)$ of the following form:
\begin{align*}
    \cF(p)^{(d)} : x^p + y^p = d,\quad d\in\mathbb{Q}^\times/\mathbb{Q}^{\times p}.
\end{align*}
The Jacobian of this curve is isogenous to the product of the Jacobian $J_{r,s,t}^{(d)}$ of
\begin{align*}
    C_{r, s, t}^{(d)} : y^p = x^r(d - x^s).
\end{align*}
Let $\chi_{r, s, t}^{(d)}$ be the Hecke character associated to $J_{r, s, t}^{(d)}$.

If $2$ is a quadratic residue modulo $p$, then one can show that there are infinitely many $d$
such that $\chi_{r, s, t}^{(d)}$ satisfies the condition of Theorem \ref{main:main}.

\begin{theorem} \label{main:Fermat}
	Let $p$ be a prime such that all places of $F$ above $2$ split in $K/F$.
	Then for each $r, s, t$ satisfying $r+s+t=p$, there are infinitely many $d$ such that $L(1, \chi_{r, s, t}^{(d)} ) \neq 0$.
\end{theorem}

We note that Diaconu and Tian \cite{DT} proved that there are infinitely many $d$ such that the $L$-function of $\cF(p)^{(d)}$ over a totally real extension $F/\bQ$ of odd degree does not vanish.

For given two newforms on $\GL_2$, it is an important problem to find a character that makes the $L$-functions of newforms twisted by the character are simultaneously nonvanishing (cf. \cite{BFH}).
As a corollary of Theorem \ref{main:Fermat}, we obtain a simultaneous nonvanishing of $L$-functions of Hecke characters $\chi_{r, s, t}$. See Corollary \ref{cor:simul} and the example followed.
In this viewpoint, Diaconu--Tian gave simultaneous nonvanishing of $\chi_{1, s, p-s-1}$ over a totally real field for all $s = 1, \cdots, p-2$.

When the primes above $2$ are inert in $K/F$, the situation becomes more complicated since
we need to check the Weil index condition.
It is involved with a computation of a certain Gauss sum. Even though we cannot get a vanishing criterion of this Gauss sum of this type, it can be numerically checked.
In particular, we give a sufficient condition for the Weil index condition when $p = 11$.

\begin{theorem} \label{main:on11}
	Let $\chi$ be a Hecke character on $K = \bQ(\zeta_{11})$ such that $\chi_2$ is a nontrivial unramified quadratic character. 
 If $\alpha$ in Theorem \ref{thm:Yang97main} satisfies $\ord_2(\alpha) \equiv 0 \pmod{2}$, then $\chi$ satisfies the Weil index condition \ref{cond:Weil}.
\end{theorem}
%For the definition of $\alpha$, see Section \ref{sec:sketch}.
There is an example of a Hecke character that is ramified at $2$ but satisfies the Weil index condition. Let
\begin{align*}
	C_A : y^2 = x^{11} + A^2, \qquad A \in \bZ
\end{align*}
be a hyperelliptic curve, and let $\chi_A$ be a Hecke character on $\bQ(\zeta_{11})$ corresponds to $C_A$.
Note that $C_A$ is a twist of $C_1$, and $\chi_A$ is ramified at $2$ (see Lemma \ref{lem:11hyperelliptic} (iii)).

\begin{theorem} \label{main:hyperelliptic}
Assume that $A$ is the $11$-th power free integer such that
\begin{itemize}
\item $(A, 22) = 1$,
\item $11^2 \mid (A^{10}-1)$,
\item $A \equiv 1 \pmod{4}$, and
\item all primes dividing $A$ split in $K/F$.
\end{itemize}
	Then $L\left(1, \chi_A \right) \neq 0$.
\end{theorem}
We remark that similar Hecke characters on $\bQ(\zeta_5)$ are considered in \cite{SY03, JPY}.

\section{Notations and the Sketch of the Proof} \label{sec:sketch}
In this section, we will introduce some notations, recall the previous results, and give an outline of this paper.

Let $K$ be the $p$-th cyclotomic field. Let $F\subseteq K$ be the maximal totally real subfield. Take $\delta = \zeta_p^{-\frac{p-1}{2}} - \zeta_p^{\frac{p-1}{2}}$ via the fixed embedding which identify $\zeta_p$ with $e^{\frac{2 \pi i}{p}}$. Let $\Delta = \delta^2 \in F$.
%We note that $\delta$ is an element satisfying $\delta^2 \in F^\times - F^{\times 2}$.
We denote by $x\mapsto\overline{x}$ the unique non-trivial element in $\Gal(K/F)$.	
%Let $W = K$ with the skew-Hermitian form $\lara{x, y} = \delta x\overline{y}$.
%Then, its isometry group is $K^1$, which is a kernel of norm from $K^\times$ to $F^\times$.
The quadratic character attached to $K/F$ is denoted by $\epsilon_{K/F}$, which can be regarded as a continuous homomorphism from $\bA_F^\times$ to $\bC$.
Let $\chi$ be a Hecke character of $K$ whose restriction to $F$ is $\epsilon_{K/F}$.
Let $\psi$ is an additive character of $\mathbb{A}_F$ given by $\psi = \prod_v \psi_v$ for $\psi_v(x) = e^{-2\pi i\lambda_v(x)}$ where 
\begin{align}\label{def:lambdav}
\xymatrix{\lambda_v : F_v \ar[r]^-{\mathrm{Tr}_{F_v/\mathbb{Q}_p}} & \bQ_p \ar[r] & \bQ_p/\bZ_p \ar[r] & \bQ/\bZ},
\end{align}
and $\psi_K := \psi \circ \mathrm{Tr}_{K/F}$.
From now on, we will denote $v$ for the prime or place of $F$ and $w$ for that of $K$.
%We sometimes clarify the field and the place, for example, $\psi_{F, v}$ and $\psi_{F, \fp}$.
We let rings act on additive characters defined on them by multiplication with arguments. For example,
\begin{align*}
\left(\frac{1}{2} \psi_{K,w}\right)(x):=\psi_{K,w}\left(\frac{1}{2}x\right).
\end{align*}
Let $G:=\ker(\Nm_{K/F})$ be the norm $1$ subgroup of $K^\times$.
Similarly we define $G_v = \ker(\Nm_{K_w/F_v})$ for a prime $v$ of $F$, and 
\begin{align*}
	\Gamma_k = \lcrc{g \in G_v : g \equiv 1 \pmod{\pi_v^k}}, \qquad
	\Gamma_k' = \lcrc{g = x + y\delta \in G_v : x \in F_v, y \in \pi_v^k \cO_{F_{v}}}.
\end{align*}
where $\pi_v$ is a uniformizer for $v$.
Given a character $\eta$ on $G$, we denote $\widetilde{\eta}(z) := \eta(z/\overline{z})$.
%The Tate local root number of $\chi\widetilde{\eta}$ is denoted by
%\begin{align*}
%	\epsilon\lbrb{\frac{1}{2}, \chi\widetilde{\eta}, \frac{1}{2}\psi_E}.
%\end{align*}
Given $\alpha \in F^\times/\Norm_{K/F} K^\times$, we denote by $\omega_{\alpha, \chi}$ a Weil representation acting on the space of Schwartz functions, and $\theta_{\phi}(\eta)$ by a theta lifting.
For details, see \cite[\S 0]{Yan97}.

The main tool of this paper to show the nonvanishing of $L$-functions is the following result of Yang.
\begin{theorem}[{\cite[Theorem 0.3]{Yan97}}] \label{thm:Yang97main}
	Let $\chi$ be a Hecke character on $K$ whose restriction to $F$ is $\epsilon_{K/F}$.
    Assume that the global root number of $\chi \widetilde{\eta}$ is $+1$ and 
    $\alpha \in F^\times/\Norm_{K/F} K^\times$ is the unique element such that
    \begin{align} \label{eqn:alpha}
        \epsilon\lbrb{\frac{1}{2}, (\chi\widetilde{\eta})_v, \frac{1}{2} \psi_v } (\chi\widetilde{\eta})_v(\delta) = \epsilon_{K/F, v}(\alpha)
    \end{align}
    holds for each place $v$ of $F$. 
    Then, there is a Schwartz function $\phi = \prod_v \phi_v$ on the adele ring $\bA_F$ and an absolute constant $c > 0$ depending on $\phi$ satisfying 
    \begin{align*}
        \frac{L\lbrb{1, \chi\widetilde{\eta}}}{L(1, \epsilon_{K/F})} = c \lara{\theta_{\phi}(\eta), \theta_{\phi}(\eta)}.
    \end{align*}
\end{theorem}

Since $\alpha \in F^\times/\Norm_{F/K}K^\times$, $\ord_v(\alpha) \equiv 0 \pmod{2}$ if and only if the representative of $\alpha$ can be choose in $\cO_{F_v}^\times$. In this case, we simply denote $\alpha \in \cO_{F_v}^\times$.

We consider the cases where $\eta$ is trivial.
Since $L(1, \epsilon_{K/F}) \neq 0$, it is enough to show $\theta_{\phi}(1)(1) \neq 0$ to show the nonvanishing of the special value of $L$-function of $\chi$.
Let
\begin{align*}
    U = \prod_v U_v
\end{align*}
be a subgroup of $\bA_F$ satisfying $G(\bA_F) = G(F) U$ and $G(F) \cap U = \lcrc{\pm 1}$
(cf. \cite[p. 277]{SY03}). Then,
\begin{align} \label{eqn:theta11}
    \theta_{\phi}(1)(1) &=
    \int_{[G]} \sum_{x \in F} \omega_{\alpha, \chi}(g) \phi(x) dg
    = \frac{1}{2} \int_U \sum_{x \in F} \omega_{\alpha, \chi}(g) \phi(x) dg
    = \frac{1}{2} \sum_{x \in F} \prod_v I_v(x)
\end{align}
where
\begin{align} \label{eqn:[G]I_v}
	[G] := G(F) \backslash G(\bA_F) \qquad \textrm{and}  \qquad I_v(x) := \int_{U_v} \omega_{\alpha, \chi, v}(g) \phi_v(x) dg.
\end{align}

Hence we need to choose the eigenfunction $\phi_v$ and study the action of Weil representation on $\phi_v$, to evaluate $I_v(x)$.
Here is a summary of the previous results of $I_v(x)$:
\begin{enumerate}
	\item When $v$ is real and $\chi$ satisfies the infinite place condition of Theorem \ref{main:main}, the eigenfuction $\phi_v$ and the action of Weil representation are given in \cite[Lemma 1.1]{Yan97}.
	\item If everything is unramified, i.e., $K/F$ is unramified at $v$, $\psi, \chi, \widetilde{\eta}$ is unramified at $v$, and $2\Delta\alpha \in \cO_{F_{v}}^\times$, then by \cite[Corollary 2.14]{Yan97} $\phi_v$ is the characteristic function of $\cO_{F_{v}}$ and the action of Weil representation is trivial.
	Therefore, $I_v(x)$ is a product of a nonzero constant and the characteristic function of $\cO_{F_v}$.
    \item If $v$ splits in $K/F$,  then by \cite[Corollary 2.14]{Yan97} $\phi_v = (\meas(\cO_{F_{v}}))^{-\frac{1}{2}}\ch(\cO_{F_{v}})$ is an eigenfunction, and the action of Weil representation is given in \cite[Corollary 2.10 (i)]{Yan97}.
    For the Hecke character attached to the hyperelliptic curves $y^2 = x^5 + A$, authors chose $\phi_v$ and showed that $I_v(x)$ is a product of a nonzero constant and a characteristic function of a compact set in \cite[\S 4]{JPY}. 
    Using a similar technique, we will show the same result on $I_v(x)$ for arbitrary Hecke character holds the condition of Theorem \ref{main:main} in Section \ref{sec:splitting and p}.
    \item If $v \nmid 2$ does not split in $K/F$, then by \cite[Corollary 2.4]{Yan97} $\phi_v$ can be taken to a unitary eigenfunction and the concrete choice of $\phi_v$ is given in \cite[Corollary 1.4, Proposition 1.5]{Yan99} for some cases.
    Furthermore, for $v \nmid 2$, the action is given in \cite[Proposition 1.2]{Yan99}.
    However, it is hard to control $I_v(x)$ (cf. \cite[(9), (10)]{JPY}).
    We do not consider this case in this paper.
    \item When $v \mid 2$ and $K = \bQ(\zeta_5)$, the eigenfunction $\phi_2$ for the Hecke character attached to the hyperelliptic curves $y^2 = x^5 + A$ and the action of Weil representation are given in \cite{SY03}.
\end{enumerate}

Therefore, the only primes $v$ of $F$ where the construction of the eigenfunction is not suggested in \cite{Yan97, Yan99, SY03, JPY} are primes above $2$ which do not split in $K/F$.
After quickly considering the real cases, the split cases and the ramified case in Section \ref{sec:splitting and p}, we mainly compute $I_v(x)$ for prime $v\mid 2$ over $\bQ(\zeta_p)$ where all primes above $2$ are inert in $K/F$ in Section \ref{sec:above 2}, under the Weil index condition \ref{cond:Weil}.
This leads to the proof of Theorem \ref{main:main}.
After that, we will give a sufficient condition for the Weil index condition where $K = \bQ(\zeta_{11})$, which is Theorem \ref{main:on11}.
As a concrete example, we will show the nonvanishing of $L$-functions of certain twists of Fermat curves and hyperelliptic curves in Section \ref{sec:Fermat}, which is Theorem \ref{main:Fermat} and Theorem \ref{main:hyperelliptic}, respectively

\section{Local computations at splitting primes and the prime above $p$} \label{sec:splitting and p}

The goal of this section is to compute $I_v(x)$ for $v$ being an infinite prime, the prime above $p$, and a prime which splits in $K/F$.
Let $\fp$ be the unique prime of $F$ above $p$ with a uniformizer $\pi_{\fp}$, and let $\fP$ be the unique prime of $K$ above $\fp$.

\begin{proposition} \label{prop:Ivaboveinf}
	Let $v$ be a real place of $F$ and let $w$ be the complex place of $K$ dividing $v$. Suppose that $\chi_w(z) = |z|_w/z$ and $\alpha$ in (\ref{eqn:alpha}) satisfies $\sigma_v(\alpha) > 0$. Then
	\begin{align*}
		\phi_v(x) = |2\sigma_w(\alpha \delta^3)|^{\frac{1}{4}} e^{-\pi |\sigma_w(\alpha \delta^3)|\sigma_v(x)^2}
	\end{align*}
	is an eigenfunction with trivial eigencharacter.
\end{proposition}
\begin{proof}
It follows from \cite[Lemma 1.1]{Yan99}. We use the normalization of \cite[Lemma 3.2]{SY03}.
\end{proof}

\begin{proposition} \label{prop:Ivabovep}
	Suppose the conductor exponent of $\chi$ at $\fp$ is $1$. Let
	\begin{align*}
		\xi = \left\{
		\begin{array}{cc}
		1 &  \textrm{if } \epsilon_{\fP}\lbrb{\frac{1}{2}, \chi_{\fP}, \frac{1}{2} \psi_{K, \fP}}\chi_{\fP}(\delta) = \epsilon_{\fP}(\alpha), \\
		\mathrm{sign} & \textrm{otherwise}.
		\end{array}
		\right.
	\end{align*}
	where $\mathrm{sign}$ means the non-trivial character on $G_{\fp}/\Gamma_1 \cong \lcrc{\pm1}$.
	Let $n = n(\psi_{\fp}) - 1 -\ord_{\fp}(\alpha)$.
	Then, $\phi_v(x) = \ch(\pi_{\fp}^{[\frac{n}{2}]}\cO_{F_{\fp}})$ is an eigenfunction with an eigencharacter $\xi$. 
\end{proposition}
\begin{proof}
	By \cite[Lemma 2.3]{Yan98}, the conductor of $\psi_{F, \fp}' := \frac{\delta \alpha}{4} \psi_{K, \fP}$ is even, and actually equals to
	\begin{align*}
		2n:=2n(\psi_{\fp}) - 2 - 2\ord_{F_{\fp}}(\alpha).
	\end{align*}
	Then, $\ch(\pi_{\fp}^{[\frac{n}{2}]}\cO_{F_{\fp}})$ is an eigenfunction with an eigencharacter $\xi$ by \cite[(1.9), Corollary 1.4 (1)]{Yan99}.
\end{proof}

\begin{proposition} \label{prop:Ivsplit}
	Let $v \mid \ff(\chi)$ be a prime of $F$ that splits in $K/F$. Then, $I_{v}(x)$ is a product of a nonzero constant depending on $\alpha$ defined in (\ref{eqn:alpha}) and a characteristic function of a compact set.
\end{proposition}
\begin{proof}
	This is proved in \cite[\S 4]{JPY}. For the reader's convenience, we reproduce some of the arguments there. Under the identification
\begin{align*}
K_v\cong\frac{F[t]}{(t^2-\Delta)}\otimes_FF_v\cong F_v\cdot\delta\oplus F_v\cdot(-\delta)
\end{align*}
we have $\delta=(1,-1)\in F_v\oplus F_v$. Let $\pi_{F_v}\in\mathcal{O}_{F_{v}}$ be a uniformizer.
%To get $\phi_v=\phi_{v,1}$ using \cite[Theorem 2.15]{Yan97}, we first compute
Recall the definition of $\rho$ given in \cite[Theorem 2.15]{Yan97}.
Then, a computation of \cite[\S 4]{JPY} gives
\begin{align*}
\rho\left(\mathrm{char}\left(1+\pi_{F_v} \mathcal{O}_{F_{v}}\right)\right)(x)&=|\alpha|_v^{\frac{1}{2}}\psi_v\left(\frac{\alpha x^2}{2}+\alpha x\right)\mathrm{meas}(\pi_{F_v}\cO_{F_{v}}) \mathrm{char}\left(\pi_{F_v}^{-2}\mathcal{O}_{F_{v}}\right)(x).
\end{align*}
Hence we get
\begin{align*}
\phi_v=\mathrm{meas}(\mathcal{O}_{F_{v}})^{-\frac{1}{2}}\mathrm{meas}(\pi_{F_v}\cO_{F_{v}}) q_v^{\frac{1}{2}}|\alpha|_v^{\frac{1}{2}}\psi_v\left(\frac{\alpha x^2}{2}+\alpha x\right)\mathrm{char}\left(\pi_{F_v}^{-2}\mathcal{O}_{F_{v}}\right)(x)
\end{align*}
also. Since $\chi_v|_{\cO_{F_{v}}^\times}$ factor through $(\cO_{F}/\fp_v)^\times$ and the restriction of $\chi$ on $F$ is quadratic, $\chi_v|_{\cO_{F_{v}}^\times} = 1$.
Therefore,
\begin{align*}
I_v(x)&=\int_{\mathcal{O}_{F_{v}}^\times}\omega_{\alpha,\chi_A,v}(g)\phi_v(x)dg\\
&=\mathrm{meas}(\mathcal{O}_{F_{v}})^{-\frac{1}{2}}\mathrm{meas}(\pi_{F_v}\cO_{F_{v}}) q_v^{\frac{1}{2}}|\alpha|_v^{\frac{1}{2}}\mathrm{char}\left(\pi_{F_v}^{-2}\mathcal{O}_{F_{v}}\right)(x)\int_{\mathcal{O}_{F_{v}}^\times}\psi_v\left(\frac{\alpha}{2}(xg)^2+\alpha(xg)\right)dg.
\end{align*}
We note that the action of Weil representation $\omega$ is described in \cite[Corollary 2.10]{Yan97}, and an explicit computation can be found in \cite[\S 4]{JPY}.
Since we can choose $\alpha$ up to $\Nm_{K/F}K^\times$, we may take the valuation of $\alpha$ as large as possible. 
Therefore, for a suitable $\alpha$, $\psi_v\left(\frac{\alpha}{2}(xg)^2+\alpha(xg)\right) = 1$ for $g \in \cO_{F_{v}}^\times$ and $x \in \pi_{F_v}^{-2} \cO_{F_{v}}$ for all $v \mid \ff(\chi)$ splits in $K/F$.
Hence
\begin{align*}
I_v|_{\pi_{F_v}^{-2}\mathcal{O}_{F_{v}}}=\mathrm{meas}(\mathcal{O}_{F_{v}})^{-\frac{1}{2}} \mathrm{meas}(\pi_{F_v}\cO_{F_{v}}) q_v^{\frac{1}{2}}|\alpha|_v^{\frac{1}{2}}\int_{\mathcal{O}_{F_{v}}^\times}dg
=\frac{\mathrm{meas}(\mathcal{O}_{F_{v}}^\times)}{\mathrm{meas}(\mathcal{O}_{F_{v}})^{\frac{1}{2}}}
\mathrm{meas}(\pi_{F_v}\cO_{F_{v}}) q_v^{\frac{1}{2}}|\alpha|_v^{\frac{1}{2}}
\end{align*}
is a non-zero constant. Therefore, there is a non-zero constant $c_v(\alpha)$ such that 
\begin{align*} \label{eqn:Ivsplit}
	I_v(x) = c_v(\alpha) \mathrm{char} (\pi_{F_v}^{-2}\cO_{F_{v}})(x),
\end{align*}
when $v \mid \ff(\chi)$ splits in $K/F$.
\end{proof}

\section{Local computations at primes above $2$} \label{sec:above 2}

\subsection{Weil index condition and the proof of Theorem \ref{main:main}}
Let $v$ be a prime of $F$ which is inert in $K/F$, and let $\iota_{\alpha, v} : G_v \to \Sp_1(F_v)$ be a map defined by
\begin{align*}
	g = x +y\delta \to \left(
	\begin{array}{cc}
	x & \Delta^2 \alpha y \\
	\frac{y}{\Delta \alpha} & x
	\end{array}
	\right).
\end{align*}
For a nontrivial additive character $\psi$ of $F_{v}$, denote by $\gamma_{F,v}(\psi)$ the Weil index of the map $x\mapsto\psi(x^{2})$ (cf. \cite[p. 161]{Wei64}). When $v\mid 2$, $\gamma_{F,v}(\psi)$ can be computed by using the following formula:
\begin{equation}\label{eqn:Weildyadic}
\gamma_{F,v}(\psi)=|\mathcal{O}_{F_{v}}/\pi_{v}^{\widetilde{m}}\mathcal{O}_{F_{v}}|^{-\frac{1}{2}}\sum_{\bar{x}\in\mathcal{O}_{F_{v}}/\pi_{v}^{\widetilde{m}}\mathcal{O}_{F_{v}}}\psi(\pi_{v}^{-\widetilde{m}-m - \ord_v(2)}x^{2}).
\end{equation}
Here, $m=m(\psi)$ is the largest integer such that $\psi|_{\pi_{v}^{-m}\mathcal{O}_{F_{v}}}=1$, 
\begin{align*}
\overline{m}:=\left\{\begin{array}{ll}
0 & \textrm{if $m$ is even,}\\
1 & \textrm{if $m$ is odd},
\end{array}\right.
\end{align*}
and $\widetilde{m}=\overline{m}+ \ord_{\pi_{v}}(2)$ (cf. \cite[p. 370]{Rao}). Given $a\in F_{v}^{\times}$, we define $a\psi:F_{v}\rightarrow\mathbb{C}^{\times}$ by $(a\psi)(x)=\psi(ax)$, and set
\begin{equation}\label{eqn:Windratio}
\gamma_{F,v}(a,\psi)=\frac{\gamma_{F,v}(a\psi)}{\gamma_{F,v}(\psi)}.
\end{equation}
It is known that 
\begin{equation}\label{eqn:square}
\gamma_{F,v}(ac^{2},\psi)=\gamma_{F,v}(a,\psi)
\end{equation}
for all $a,c\in F_{v}^{\times}$ (cf. \cite[Theorem A.4]{Rao}).

Now, for $g=x+y\delta\in G_{v}$ where $x, y \in F_v$
\begin{align*}
	\mu_{\alpha, \chi, v}(g) := \chi_v(-\delta y (g-1))\gamma_{F, v}(\alpha y (1-x) \psi_{v}).
\end{align*}
Denote by $r_v$ the Rao's standard section \cite[\S 3]{Rao}. Then
\begin{equation*} \label{eqn:Weilrepn}
	\omega_{\alpha, \chi, v}(g)\phi_v = \mu_{\alpha, \chi, v}(g)r_v(\iota_{\alpha, v}(g))\phi_v.
\end{equation*}
Denote
\begin{align*}
	\xi_{\alpha, \chi, v}(g) = \left\{
	\begin{array}{ll}
	\mu_{\alpha, \chi, v}(g) \gamma_{F, v}(2\Delta \alpha xy \psi_v) & \textrm{if } g \in \Gamma_1 \\
		\mu_{\alpha, \chi, v}(g)  & \textrm{if } g \in G_{v} - \Gamma_1 
	\end{array}
	\right.
\end{align*}
so that, by \cite[(5.3)]{SY03}, we have
\begin{align*}
	\omega_{\alpha, \chi, v}(g) \phi_v(u) = \xi_{\alpha, \chi, v}(g) r_{v}(\iota_{\alpha, v}w) \widehat{\phi}_v(-u)
\end{align*}
where $\widehat{\phi}_v$ is the Fourier transform of $\phi_v$.

In this section, we will consider primes $v \mid 2$ which are inert in $K/F$.
As in Section \ref{sec:sketch}, we define $G_v$ as the kernel of norm from $K_w$ to $F_v$ and
\begin{align*}
	\Gamma_k = \lcrc{g \in G_v : g \equiv 1 \pmod{\pi_v^k}}, \qquad
	\Gamma_k' = \lcrc{g = x + y\delta \in G_v : x \in F_v, y \in \pi_v^k \cO_{F_{v}}}.
\end{align*}
Note that for each $g \in G_v$, there is unique $x, y \in \cO_{F_{v}}$ such that $g = x + y\delta$.

\begin{lemma} \label{lem:SY5.35.4}
	Assume that $\alpha  \in \cO_{F_{v}}^\times$. Let $f_{\beta} = \ch(\beta + \cO_{F_{v}})$ for $\beta \in \frac{1}{2}\cO_{F_{v}}$. Then,
	\begin{align*}
		\omega_{\alpha, \chi, v}(g)f_{\beta} = \xi_{\alpha, \chi, v}(g) \psi_{v}\lbrb{\frac{\Delta^2\alpha y}{2x}\beta^2}f_{\beta}
	\end{align*}
	for all $g = x + y\delta \in \Gamma_1$. Furthermore, $\xi_{\alpha, \chi, v}(g)$ is a character on $\Gamma_1$ which does not depend on $\alpha$.
\end{lemma}
\begin{proof} The first assertion follows from \cite[Lemma 5.3]{SY03}. For the rest, see \cite[Lemma 5.4]{SY03}.
\end{proof}

The next lemma is an analogue of \cite[Lemma 5.6]{SY03}. For the reader's convenience, we present a detailed proof. The point is to show \eqref{lem:Weilunity-reduction} which converts a quadratic term in the integrand to a linear term.

\begin{lemma}\label{lem:Weilunity}
Assume that $\alpha\in\mathcal{O}_{F_{v}}^{\times}$.
Let $\kappa(F_v)$ be the residue field of $F_v$ and let $k_v$ be the residue degree $[\kappa(F_v) : \bF_2]$. 
Let
\begin{align*}
	\beta=\frac{1}{2}\left(\alpha(\zeta_{p}^{\frac{p-1}{2}}+\zeta_{p}^{-\frac{p-1}{2}})^{3}\right)^{2^{k_v-1}-1}\in\frac{1}{2}\mathcal{O}_{F_{v}}.
\end{align*}
Then for $1\leq j<p$,
\begin{align*}
\omega_{\alpha,\chi,v}(\zeta_{p}^{j})f_{\beta}(u)=2^{-\frac{k_v}{2}}\chi_{v}(-&\delta y_{j} (\zeta_{p}^{j}-1))\gamma_{F, v}(\alpha y_{j} (1-x_{j}) \psi_{v})\\
&\times\psi_{v}\left(\frac{\Delta\alpha x_{j}}{2y_{j}}\left(u^{2}-2\beta\frac{u}{x_{j}}+\beta^{2}\right)\right)\mathrm{char}\left(\frac{1}{2}\mathcal{O}_{F_{v}}\right)(u),
\end{align*}
where $\zeta_{p}^{j}=x_{j}+y_{j}\delta$.
\end{lemma}

\begin{proof}
Fix $1\leq j<p$. It is known that
\begin{align*}
r_{v}(m(a))\phi(u)=|a|_{v}^{\frac{1}{2}}\phi(au), &\qquad
r_{v}(n(b))\phi(u)=\psi_{v}\left(\frac{1}{2}bu^{2}\right)\phi(u),\qquad
r_{v}(w)\phi(u)=\widehat{\phi}(u),\\
r_{v}(m(a)n(b)wm(a')n(b'))&=r_{v}(m(a))r_{v}(n(b))r_{v}(w)r_{v}(m(a'))r_{v}(n(b')),
\end{align*} 
where
$$m(a)=\begin{pmatrix}a & 0 \\ 0 & a^{-1}\end{pmatrix},\quad n(b)=\begin{pmatrix}1 & b \\ 0 & 1\end{pmatrix}\quad(a,b\in F_{v}, a\neq 0)$$
(cf. \cite[p.281]{SY03}).
Thus,
\begin{align*}
r_{v}(\iota_{\alpha,v}(\zeta_{p}^{j}))f_{\beta}(u)&=r_{v}\left(m\left(\frac{\Delta\alpha}{y_{j}}\right)n\left(\frac{x_{j}y_{j}}{\Delta\alpha}\right)wn\left(\frac{\Delta\alpha x_{j}}{y_{j}}\right)\right)f_{\beta}(u)\\
&=r_{v}\left(m\left(\frac{\Delta\alpha}{y_{j}}\right)\right)r_{v}\left(n\left(\frac{x_{j}y_{j}}{\Delta\alpha}\right)\right)\int_{F_{v}}f_{\beta}(t)\psi_{v}\left(\frac{\Delta\alpha x_{j}}{2y_{j}}t^{2}\right)\psi_{v}(-ut)dt\\
&=r_{v}\left(m\left(\frac{\Delta\alpha}{y_{j}}\right)\right)r_{v}\left(n\left(\frac{x_{j}y_{j}}{\Delta\alpha}\right)\right)\int_{\beta+\mathcal{O}_{F_{v}}}\psi_{v}\left(\frac{\Delta\alpha x_{j}}{2y_{j}}t^{2}\right)\psi_{v}(-ut)dt\\
&=r_{v}\left(m\left(\frac{\Delta\alpha}{y_{j}}\right)\right)r_{v}\left(n\left(\frac{x_{j}y_{j}}{\Delta\alpha}\right)\right)\psi_{v}\left(\frac{\Delta\alpha x_{j}}{2y_{j}}\beta^{2}-\beta u\right)\\
&\quad\quad\quad\quad\quad\quad\quad\quad\quad\quad\quad\quad\quad\quad\quad\times\int_{\mathcal{O}_{F_{v}}}\psi_{v}\left(\frac{\Delta\alpha x_{j}}{2y_{j}}t^{2}+\frac{\Delta\alpha x_{j}}{y_{j}}\beta t-ut\right)dt\\
&=\left|\frac{\Delta\alpha}{y_{j}}\right|_{v}^{\frac{1}{2}}\psi_{v}\left(\frac{\Delta\alpha x_{j}}{2y_{j}}\left(u^{2}-2\beta\frac{u}{x_{j}}+\beta^{2}\right)\right)\\
&\quad\quad\quad\quad\quad\quad\quad\quad\quad\quad\quad\quad\quad\times\int_{\mathcal{O}_{F_{v}}}\psi_{v}\left(\frac{\Delta\alpha x_{j}}{2y_{j}}t^{2}+\frac{\Delta\alpha x_{j}}{y_{j}}\beta t-\frac{\Delta\alpha}{y_{j}}ut\right)dt.
\end{align*}
Since $\alpha,\Delta\in\mathcal{O}_{F_{v}}^{\times}$ and $y_j\in\frac{1}{2}\mathcal{O}_{F_{v}}^{\times}$, $|\Delta\alpha/y_j|_{v}=2^{-k_v}$. It turns out that
$$r_{v}(\iota_{\alpha,v}(\zeta_{p}^{j}))f_{\beta}(u)=2^{-\frac{k_v}{2}}\psi_{v}\left(\frac{\Delta\alpha x_{j}}{2y_{j}}\left(u^{2}-2\beta\frac{u}{x_{j}}+\beta^{2}\right)\right)\int_{\mathcal{O}_{F_{v}}}\psi\left(\frac{\Delta\alpha x_{j}}{2y_{j}}t^{2}+\frac{\Delta\alpha x_{j}}{y_{j}}\beta t-\frac{\Delta\alpha}{y_{j}}ut\right)dt.$$
We know that $t^{2^{k_v}}\equiv t\pmod{\pi_{v}}$ for $t\in\mathcal{O}_{F_{v}}$, and it follows that
\begin{align*}
&\mathrm{Tr}_{\kappa(F_{v})/\mathbb{F}_{2}}\left(\frac{\Delta\alpha x_{j}}{y_{j}}t^{2}~(\mathrm{mod}~\pi_{v})\right)\\
&=\left(\frac{\Delta\alpha x_{j}}{y_{j}}\right)t^{2}+\left(\frac{\Delta\alpha x_{j}}{y_{j}}\right)^{2}t^{2^{2}}+\cdots+\left(\frac{\Delta\alpha x_{j}}{y_{j}}\right)^{2^{k_v-2}}t^{2^{k_v-1}}+\left(\frac{\Delta\alpha x_{j}}{y_{j}}\right)^{2^{k_v-1}}t\pmod{2}\\
&=\left(\frac{\Delta\alpha x_{j}}{y_{j}}\right)^{2^{k_v-1}}t+\left(\frac{\Delta\alpha x_{j}}{y_{j}}\right)t^{2}+\left(\frac{\Delta\alpha x_{j}}{y_{j}}\right)^{2}t^{2^{2}}+\cdots+\left(\frac{\Delta\alpha x_{j}}{y_{j}}\right)^{2^{k_v-2}}t^{2^{k_v-1}}\pmod{2}\\
&=\mathrm{Tr}_{\kappa(F_{v})/\mathbb{F}_{2}}\left(\left(\frac{\Delta\alpha x_{j}}{y_{j}}\right)^{2^{k_v-1}}t~(\mathrm{mod}~\pi_{v})\right).
\end{align*}
This means that
\begin{equation*}
    \mathrm{Tr}_{F_{v}/\mathbb{Q}_{2}}\left(\frac{\Delta\alpha x_{j}}{y_{j}}t^{2}-\left(\frac{\Delta\alpha x_{j}}{y_{j}}\right)^{2^{k_{v}}-1}t\right)\in 2\mathbb{Z}_{2}.
\end{equation*}
Following notation \eqref{def:lambdav}, we have $\psi_v=\psi_{\mathbb{Q}_2}\circ\mathrm{Tr}_{F_v/\mathbb{Q}_2}$ and $\psi_{\mathbb{Q}_2}$ vanishes on $\mathbb{Z}_2$. Consequently,
\begin{align}\label{lem:Weilunity-reduction}
\psi_{v}\left(\frac{\Delta\alpha x_{j}}{2y_{j}}t^{2}\right)=\psi_{v}\left(\frac{1}{2}\left(\frac{\Delta\alpha x_{j}}{y_{j}}\right)^{2^{k_v-1}}t\right).
\end{align}
Using this, we get
$$\int_{\mathcal{O}_{F_{v}}}\psi_{v}\left(\frac{\Delta\alpha x_{j}}{2y_{j}}t^{2}+\frac{\Delta\alpha x_{j}}{y_{j}}\beta t-\frac{\Delta\alpha}{y_j}ut\right)dt=\int_{\mathcal{O}_{F_{v}}}\psi_{v}\left(\frac{\Delta\alpha x_{j}}{y_{j}}\left(\frac{1}{2}\left(\frac{\Delta\alpha x_{j}}{y_{j}}\right)^{2^{k_v-1}-1}+\beta-\frac{u}{x_{j}}\right)t\right)dt.$$
The right-hand side is nontrivial if and only if
$$\frac{1}{2}\left(\frac{\Delta\alpha x_{j}}{y_{j}}\right)^{2^{k_v-1}-1}+\beta-\frac{u}{x_{j}}\in\mathcal{O}_{F_{v}},$$
which in turn gives 
\begin{align*}
\int_{\mathcal{O}_{F_{v}}}\psi_{v}\left(\frac{\Delta\alpha x_{j}}{2y_{j}}t^{2}+\frac{\Delta\alpha x_{j}}{y_{j}}\beta t-\frac{\Delta\alpha}{y_j}ut\right)dt=\mathrm{char}\left(\left(\frac{1}{2}\left(\frac{\Delta\alpha x_{j}}{y_{j}}\right)^{2^{k_v-1}-1}+\beta\right)x_{j}+\frac{1}{2}\mathcal{O}_{F_{v}}\right)(u).
\end{align*}
Recall that $\delta=\zeta_{p}^{-\frac{p-1}{2}}-\zeta_{p}^{\frac{p-1}{2}}$ and $\Delta=\delta^{2}$. Since $\zeta_{p}^{j}=x_{j}+y_{j}\delta$ and $y_{j}\in\frac{1}{2}\mathcal{O}_{F_{v}}^{\times}$, we have
\begin{align*}
\frac{\Delta\alpha x_{j}}{y_{j}}=\alpha\delta^2\left(\delta+\frac{\zeta_p^j}{y_j}\right)\equiv\alpha\left(\zeta_{p}^{\frac{p-1}{2}}+\zeta_{p}^{-\frac{p-1}{2}}\right)^{3}\pmod{2\mathcal{O}_{F_{v}}}.
\end{align*}
Hence
\begin{align*}
&\quad\int_{\mathcal{O}_{F_{v}}}\psi_{v}\left(\frac{\Delta\alpha x_{j}}{2y_{j}}t^{2}+\frac{\Delta\alpha x_{j}}{y_{j}}t\beta-\frac{\Delta\alpha}{y_{j}}ut\right)dt\\
&=\mathrm{char}\left(\left(\frac{1}{2}\left(\Delta\alpha\frac{x_{j}}{y_{j}}\right)^{2^{k_v-1}-1}+\beta\right)x_j+\frac{1}{2}\mathcal{O}_{F_{v}}\right)(u)\\
&=\mathrm{char}\left(\left(\frac{1}{2}\left(\alpha(\zeta_{p}^{\frac{p-1}{2}}+\zeta_{p}^{-\frac{p-1}{2}})^{3}\right)^{2^{k_v-1}-1}+\beta\right)x_j+\frac{1}{2}\mathcal{O}_{F_{v}}\right)(u)=\mathrm{char}\left(\frac{1}{2}\mathcal{O}_{F_{v}}\right)(u).
\end{align*}
\end{proof}

Now, we introduce the Weil index condition. Note that we use the bijection
\begin{align*}
\xymatrix{\mathbb{P}^1(F_v) \ar[r] & G_v & [a] \ar@{|->}[r] & \displaystyle g_a:=\frac{a+\delta}{a-\delta}}.
\end{align*}

\begin{condition} \label{cond:Weil}
Let $\beta$ be as in Lemma \ref{lem:Weilunity}, let $u=2\beta$, and let $\chi$ be a Hecke character with an element $\alpha$ satisfying (\ref{eqn:alpha}). \\
(1) We say that $\chi$ satisfies the Weil index condition for $\Gamma_1$ if
\begin{equation} \label{eqn:Weilcondat1}
\xi_{\alpha, \chi, v}(g) = \exp\left(\frac{\pi i}{2}T\right)
 \qquad \textrm{where} \qquad T := \Tr_{F_v/\bQ_2}\lbrb{ \frac{a \alpha u^2 \Delta^2}{a^2+\Delta} }, 
\end{equation}
holds for all $g=g_{a}\in\Gamma_{1}$. 
% where
% \begin{align*}
% 	T := \Tr_{F_v/\bQ_2}\lbrb{ \frac{a \alpha u^2 \Delta^2}{a^2+\Delta} }.
% \end{align*}
We note that the exponential is well-defined since $T \in \bQ_2$ and $\exp(2\pi i)=1$.

\noindent
(2) Let $\zeta_p^j = x_j + \delta y_j$. We say that $\chi$ satisfies the Weil index condition for roots of unity if
\begin{align} \label{eqn:Weilcondatzeta}
2^{-\frac{k_v}{2}}\sum_{j=1}^{p-1}\chi_{v}(-\delta y_{j}(\zeta_{p}^{j}-1))\gamma_{F, {v}}(\alpha y_{j}(1-x_{j})\psi_{v})\psi_{v}\left(\frac{\Delta\alpha(x_{j}-1)}{y_{j}}\beta^{2}\right) \neq -1.
\end{align}
\end{condition}

\begin{remark} \label{rmk:Weilcond roots}
Note that $\psi_v$ is a unitary character. By \cite[Theorem A.1]{Rao}, $\gamma_{F,v}$ is of absolute value $1$. If $\chi_v$ is unitary as well, then the left hand side of \eqref{eqn:Weilcondatzeta} has absolute value at most $2^{-\frac{k_v}{2}}(p-1)$.
Hence if the residue degree of $2$ in $K/\bQ$ is sufficiently large, then we do not need to check the Weil index conditions for roots of unity.
\end{remark}

The following proposition provides a sufficient condition for the group $\Gamma_{1}$ to act trivially on $f_{\beta}$. It will be useful when we compute the integral $I_{v}$ defined in \eqref{eqn:[G]I_v} later.

\begin{proposition}\label{prop:Weilgammaone}
	Assume that $\alpha \in \cO_{F_v}^\times$. Let $\beta$ be as in Lemma \ref{lem:Weilunity}, let $H\leq \Gamma_{1}$, and let $\{g_{i}\}$ be a set of coset representatives of $H$ in $\Gamma_{1}$. If $\chi$ satisfies (\ref{eqn:Weilcondat1}) for all $g_i$ and for all elements of $H$, then, $f_{\beta}$ is an eigenfunction of $\Gamma_{1}$ with trivial eigencharacter.
\end{proposition}
\begin{proof}
Let $\alpha\in\mathcal{O}_{F_v}^{\times}$ and let $\beta\in\frac{1}{2}\mathcal{O}_{F_{v}}$. By Lemma \ref{lem:SY5.35.4}, $f_{\beta}$ is an eigenfunction of $\Gamma_{1}$ with eigencharacter
\begin{equation*}\label{eq-one}
\xi_{\alpha,\chi, v}(g)\psi_{v}\left(\frac{\Delta^{2}\alpha y}{2x}\beta^{2}\right)\quad\quad\quad(g=x+y\delta\in\Gamma_{1}).
\end{equation*}
If $g=g_a=x+y\delta\in\Gamma_1$, then
\begin{align} \label{eqn:xy for ga}
x=\frac{a^{2}+\Delta}{a^{2}-\Delta},\quad y=\frac{2a}{a^{2}-\Delta}
\end{align}
which in turn implies that
\begin{align*}
\frac{\Delta^{2}\alpha y}{2x}\beta^{2}=\frac{a\alpha\Delta^{2}}{a^{2}+\Delta}\beta^{2}\in\frac{1}{4}\mathcal{O}_{F_{v}}
\end{align*}
If we write $\beta=u/2$  for $u\in\mathcal{O}_{F_{v}}$, then with $T$ as in Condition \ref{cond:Weil}, we have
\begin{align*}
\psi_{v}\left(\frac{\Delta^{2}\alpha y}{2x}\beta^{2}\right):=\exp\left(-2\pi i\mathrm{Tr}_{F_{v}/\mathbb{Q}_{2}}\left(\frac{\Delta^{2}\alpha y}{2x}\beta^{2}\right)\bmod\mathbb{Z}_2\right)
=\exp\left(-\frac{\pi i}{2}T\bmod\mathbb{Z}_2\right)=i^{-T}
\end{align*}
Since $\chi_{v}$ satisfies \eqref{eqn:Weilcondat1}, 
$$
\xi_{\alpha,\chi, v}(g)\psi_{v}\left(\frac{\Delta^{2}\alpha y}{2x}\beta^{2}\right)=1
$$
for all $g=g_{a}\in\{g_{i}\}\cup H$. From the first part of Lemma \ref{lem:SY5.35.4}, we deduce that
\begin{align*}
\xymatrix{\Gamma_1 \ar[r] & \mathbb{C}^\times & g=x+y\delta \ar@{|->}[r] & \displaystyle\xi_{\alpha,\chi,v}(g)\psi_{v}\left(\frac{\Delta^{2}\alpha y}{2x}\beta^{2}\right)}
\end{align*}
is a character. Hence $f_{\beta}$ is an eigenfunction of $\Gamma_{1}$ with trivial eigencharacter.
\end{proof}

For $f_i : F_v \to \bC$, we define
\begin{align*}
    \lara{f_1, f_2} := \int_{F_v} f_1 \overline{f_2}.
\end{align*}

\begin{proposition} \label{prop:Ivat2}
	Suppose that $\alpha \in \cO_{F_v}^\times$ and $\chi_v$ satisfies the Weil index conditions \eqref{eqn:Weilcondat1} and \eqref{eqn:Weilcondatzeta}.
Let $\beta$ be as in Lemma \ref{lem:Weilunity} and let $f_\beta$ be as in Lemma \ref{lem:SY5.35.4}.
If we choose $\phi_{v}=f_{\beta}$, then $I_v(x)=\mathrm{meas}(\Gamma_{1})\phi_{v}(x)$ and
\begin{align*}
C_v:=\int_{G_{v}=\Gamma_{1}\times\langle\zeta_{p}\rangle}\langle\omega_{\alpha,\chi,v}(g)\phi_{v},\phi_{v}\rangle dg\neq0.
\end{align*}
\end{proposition}
\begin{proof}
Since $\chi_{v}$ satisfies the Weil index condition \eqref{eqn:Weilcondat1},
$f_{\beta}$ is an eigenfunction of $\Gamma_1$ with trivial eigencharacter by Propsition \ref{prop:Weilgammaone} for $H = \Gamma_1$.
Hence,
\begin{align*}
I_{v}(x)&=\int_{\Gamma_{1}}\omega_{\alpha,\chi,v}(g)\phi_{v}(x)dg=\int_{\Gamma_{1}}\omega_{\alpha,\chi,v}(g)f_{\beta}(x)dg=\int_{\Gamma_{1}}f_{\beta}(x)dg=\mathrm{meas}(\Gamma_{1})f_{\beta}(x).
\end{align*}

On the other hand,
\begin{align*}
C_{v}&=\int_{G_{v}=\Gamma_{1}\times\langle\zeta_{p}\rangle}\langle\omega_{\alpha,\chi,v}(g)\phi_{v},\phi_{v}\rangle dg=\int_{\Gamma_{1}}\langle\omega_{\alpha,\chi,v}(g)f_{\beta},f_{\beta}\rangle dg\cdot\sum_{j=0}^{p-1}\langle\omega_{\alpha,\chi,v}(\zeta_{p}^{j})f_{\beta},f_{\beta}\rangle\\
&=\int_{\Gamma_{1}}\langle f_{\beta},f_{\beta}\rangle dg\cdot\sum_{j=0}^{p-1}\langle\omega_{\alpha,\chi,v}(\zeta_{p}^{j})f_{\beta},f_{\beta}\rangle=\mathrm{meas}(\Gamma_{1})\left(1+\sum_{j=1}^{p-1}\langle\omega_{\alpha,\chi,v}(\zeta_{p}^{j})f_{\beta},f_{\beta}\rangle\right).
\end{align*}
Fix $j=1,\ldots,p-1$ and write $\zeta_{p}^{j}=x_{j}+y_{j}\delta$. Since $\beta\in\frac{1}{2}\mathcal{O}_{F_{v}}$, we have $\beta+\mathcal{O}_{F_{v}}\subset\frac{1}{2}\mathcal{O}_{F_{v}}$. 
By Lemma \ref{lem:Weilunity},
\begin{align*}
\langle\omega_{\alpha,\chi,v}(\zeta_{p}^{j})f_{\beta},f_{\beta}\rangle
&=2^{-\frac{k_v}{2}}\xi_{\alpha,\chi,v}(\zeta_{p}^{j})\int_{F_{v}}\psi_{v}\left(\frac{\Delta\alpha x_{j}}{y_{j}}\left(u^{2}-2\beta\frac{u}{x_{j}}+\beta^{2}\right)\right)\mathrm{char}\left(\frac{1}{2}\mathcal{O}_{F_{v}}\right)(u)\overline{f_{\beta}(u)}du\\
&=2^{-\frac{k_v}{2}}\xi_{\alpha,\chi,v}(\zeta_{p}^{j})\int_{\beta+\mathcal{O}_{F_{v}}}\psi_{v}\left(\frac{\Delta\alpha x_{j}}{y_{j}}\left(u^{2}-2\beta\frac{u}{x_{j}}+\beta^{2}\right)\right) du\\
&=2^{-\frac{k_v}{2}}\xi_{\alpha,\chi,v}(\zeta_{p}^{j})\psi_{v}\left(\frac{\Delta\alpha(x_{j}-1)}{y_{j}}\beta^{2}\right)\int_{\mathcal{O}_{F_{v}}}\psi_{v}\left(\frac{\Delta\alpha x_{j}}{2y_{j}}u^{2}+\frac{\Delta\alpha x_{j}}{y_{j}}u\beta-\frac{\Delta\alpha}{y_{j}}\beta u\right)du\\
&=2^{-\frac{k_v}{2}}\xi_{\alpha,\chi,v}(\zeta_{p}^{j})\psi_{v}\left(\frac{\Delta\alpha(x_{j}-1)}{y_{j}}\beta^{2}\right)\mathrm{char}\left(\frac{1}{2}\mathcal{O}_{F_{v}}\right)(\beta)\\
&=2^{-\frac{k_v}{2}}\xi_{\alpha,\chi,v}(\zeta_{p}^{j})\psi_{v}\left(\frac{\Delta\alpha(x_{j}-1)}{y_{j}}\beta^{2}\right).
\end{align*}
It turns out that
\begin{align*}
C_{v}&=\mathrm{meas}(\Gamma_{1})\left(1+2^{-\frac{k_v}{2}}\sum_{j=1}^{p-1}\xi_{\alpha,\chi,v}(\zeta_{p}^{j})\psi_{v}\left(\frac{\Delta\alpha(x_{j}-1)}{y_{j}}\beta^{2}\right)\right)\\
&=\mathrm{meas}(\Gamma_{1})\left(1+2^{-\frac{k_v}{2}}\sum_{j=1}^{p-1}\chi_{v}(-\delta y_{j}(\zeta_{p}^{j}-1))\gamma_{F,v}(\alpha y_{j}(1-x_{j})\psi_{v})\psi_{v}\left(\frac{\Delta\alpha(x_{j}-1)}{y_{j}}\beta^{2}\right)\right).
\end{align*}
Since $\chi_{v}$ satisfies the condition \eqref{eqn:Weilcondatzeta}, we conclude that $C_{v}\neq 0$.

\end{proof}

\begin{proof}[Proof of Theorem \ref{main:main}]
Let
\begin{align*}
	U = \prod_{\substack{v, \textrm{ nonsplit } }} G_{v} \times \prod_{v, \textrm{ split }} \cO_{F_v}^\times.
\end{align*}
 Then by Theorem \ref{thm:Yang97main} and (\ref{eqn:theta11}), we have
\begin{align*}
	L\lbrb{1, \chi} = c\left|\frac{1}{2} \sum_{x \in F} \prod_v I_v(x)\right|^2
\end{align*}
for some nonzero constant $c$. By Proposition \ref{prop:Ivabovep}, Proposition \ref{prop:Ivsplit} and Proposition \ref{prop:Ivat2}, there is a nonzero constant $c_v$ and a compact subset $X_v$ in $F_v$ such that
\begin{align*}
	I_v(x) = c_v \ch( X_v )(x),
\end{align*}
where $v$ is a finite place of $F$. Therefore,
\begin{align*}
	L\lbrb{1, \chi} = \frac{c}{2}\cdot \lbrb{ \prod_{v, \textrm{ finite}} c_v } \cdot
	\left| \sum_{x \in X} \prod_{v, \textrm{real} } \phi_{v}(x)  \right|^2
\end{align*}
where $X = \cap_v (X_v \cap F)$. 
By Proposition \ref{prop:Ivaboveinf}, for each real place $v$ and $w | v$ we have
\begin{align*}
		\phi_v(x) = |2\sigma_w(\alpha \delta^3)|^{\frac{1}{4}} e^{-\pi |\sigma_w(\alpha \delta^3)|\sigma_v(x)^2}.
\end{align*}
Hence $\phi_v(x)$ is positive whenever $x \in F$.
Consequently, we have $L\left(1, \chi\right) \neq 0$.
\end{proof}

\subsection{Proof of Theorem \ref{main:on11}}

Now we assume that $2$ is inert in $K/\mathbb{Q}$. To prove Theorem \ref{main:on11} we have to check the Weil index condition \eqref{eqn:Weilcondat1} for $\Gamma_1$. First of all, in Proposition \ref{prop:Gammathree}, we will show that the Weil index condition holds for $\Gamma_3$. Taking $H=\Gamma_3$ in Proposition \ref{prop:Weilgammaone}, we deduce that it suffices to check the Weil index condition for the set of coset representatives of $\Gamma_3$ in $\Gamma_1$. Next, we study the structure of $\Gamma_1/\Gamma_3$. Finally, we study the case of $K = \bQ(\zeta_{11})$ and prove Theorem \ref{main:on11}.

\begin{lemma}[{\cite[Lemma 5.1]{SY03}}] \label{lem:SYlem5.1}
	Let $g_a = (a + \delta)/(a - \delta)$. Then, a map $a \to g_a$ gives a bijection between $\mathbb{P}^1(F_2)$ and $G_2$. Also,
	\begin{enumerate}
	\item $g_a \in \Gamma_1$ if and only if $a^2 \not\equiv \Delta \pmod{4}$.
	\item $g_a \in \Gamma_1 - \Gamma_2'$ if and only if $a \in \cO_{F_{2}}^\times$ and $a^2 \not\equiv \Delta \pmod{4}$.
	\item For $k>1$, $g_a \in \Gamma_k' - \Gamma_k$ if and only if $a \in 2^{k-1}\cO_{F_{2}}$.
	\item For $k>1$, $g_a \in \Gamma_k$ if and only if $a^{-1} \in 2^{k-1}\cO_{F_{2}}$.
	\end{enumerate}
\end{lemma}

\begin{proposition}\label{prop:Gammathree}
Let $\beta$ be as in Lemma \ref{lem:Weilunity}, let $f_\beta$ be as in Lemma \ref{lem:SY5.35.4},
 and let $\chi$ be a Hecke character of $K$. If $\chi_{2}(-\delta)=1$ and $\chi_{2}(2)=-1$, then $f_{\beta}$ is invariant under the action of $\Gamma_{3}$. Equivalently, \eqref{eqn:Weilcondat1} holds for every $g\in\Gamma_{3}$.
\end{proposition}

\begin{proof}
By using Lemma \ref{lem:SYlem5.1}, we may write $g=g_a$. By using \cite[Theorem A.4, Corollary A.5]{Rao}, we can deduce (cf. \cite[Lemma 5.4]{SY03})
\begin{align} \label{eqn:xi rep}
  \xi_{\alpha,\chi,2}(g)=\chi_{2}(-\delta y(g-1))(a\Delta,a^{2}+\Delta)_{F,2}\gamma_{F,2}(a^{2}+\Delta,\psi).
\end{align}
Since $\chi_{2}(-\delta)=1$ and $\chi_{2}(2)=-1$,
$$\chi_{2}(-\delta y(g-1))=(-1)^{\mathrm{ord}_{2}(g-1)+\mathrm{ord}_{2}(y)}.$$
By Lemma \ref{lem:SYlem5.1}, $1+\frac{\Delta}{a^{2}}\equiv 1\pmod{16}$. Thus, by the local square theorem, $1+\frac{\Delta}{a^{2}}$ is a square in $F_{2}$, and so $a^{2}+\Delta$ is a square in $F_{2}$. Since $(a,b)_{F,2}\in\{\pm 1\}$ and $(a,bc)_{F,2}=(a,b)_{F,2}(a,c)_{F,2}$ for all $a,b,c\in F_{2}^{\times}$, $(a\Delta,a^{2}+\Delta)_{F,2}=1$. Furthermore, by \eqref{eqn:Windratio} and \eqref{eqn:square}, $\gamma_{F,2}(a^{2}+\Delta,\psi)=1$.
On the other hand, $\frac{a\alpha u^{2}\Delta^{2}}{a^{2}+\Delta}\in 4\mathcal{O}_{F_{2}}$ by Lemma \ref{lem:SYlem5.1}. Hence $T = \Tr_{F_v/\bQ_2}\lbrb{ \frac{a \alpha u^2 \Delta^2}{a^2+\Delta} }$ is in $4\mathbb{Z}_{2}$.
\end{proof}

By the \emph{Galois ring}, we mean the commutative ring of characteristic $p^s$ with $p^{ms}$ elements.
It is known that such a ring is isomorphic to $(\bZ/p^s\bZ)[x]/(h(x))$ for any irreducible polynomial $h$ over $\bZ/p^s\bZ$ whose degree is $m$ (cf. \cite[Chapter 14]{Wan}). Let $\xi$ be a zero of $h$. Then an element of the Galois ring can be written as
\begin{align*}
	a_0 + a_1 \xi + a_2 \xi^2 + \cdots + a_{m-1} \xi^{m-1}, \qquad a_i \in \bZ/p^s\bZ.
\end{align*}
We recall a result on the structure of the unit group of the Galois ring.

\begin{theorem}[{cf. \cite[Theorem 14.11]{Wan}}] \label{thm:Galoisunit}
	Let $R$ be a Galois ring over $\bZ/p^s\bZ$ with $p^{ms}$-elements. Then the unit group of $R$ is the direct product
\begin{align*}
R^\times=G_1\times G_2
\end{align*}
where $G_1$ is the cyclic group of order $p^m-1$ and $G_2$ is a group of order $p^{(s-1)m}$ such that
	\begin{enumerate}
		\item if $p$ is odd or if $p=2$ and $s\leq2$, then $G_2$ is a direct product of $m$ cyclic groups each of order $p^{s-1}$.
		\item if $p=2$ and $s\geq3$, then $G_2$ is a direct product of a cyclic group of order $2$, a cyclic group of order $2^{s-2}$ and $m-1$ cyclic groups each of order $2^{s-1}$.
	\end{enumerate}	 
	In (2), the generators of $G_2$ are
	\begin{itemize}
		\item $1 + 2 + \cdots + 2^{s-2}$ of order $2$,
		\item $1 + 4b$ of order $2^{s-2}$ where $b$ is a lifting of any element of $\bF_{2^m}\backslash N$ for $N = \lcrc{a^2 + a : a \in \bF_{2^m}}$,
		\item $1 + 2\xi^i$ for $i = 1, \cdots, m-1$ of order $2^{s-1}$.
	\end{itemize}	  
\end{theorem}

Fix a local field and let $\cO$ be its ring of integers, $\fm$ the maximal ideal of $\cO$
 and $U^{(n)} := 1 + \fm^{n}$.
Then there is a natural isomorphism $\cO^\times/U^{(n)} \cong (\cO/\fm^n)^\times$.
%\begin{align*}
%\xymatrix{
%0 \ar[r] & U^{(n)} \ar[r] & \cO^\times \ar[r] & \lbrb{\frac{\cO}{\fm^n}}^\times \ar[r] & 0
%}
%\end{align*}
%is exact. Hence we obtain a commutative diagram
%\begin{align*}
%\xymatrix{
%1 \ar[r]  & U^{(n+r)} \ar[r] \ar[d] & U^{(n)} \ar[r] \ar[d] & \frac{U^{(n)}}{U^{(n+r)}} \ar[r] \ar[d]^{\gamma} & 0\\
%1 \ar[r] & U^{(n+r)} \ar[r] & \cO^\times \ar[r]  & \lbrb{\frac{\cO}{\fm^{n+r}}}^\times \ar[r] & 0 
%}
%\end{align*}
%Applying snake lemma to the above sequences, we have $\ker \gamma = 1$ and the cokernel of $\gamma$ is isomorphic to $\cO^\times/U^{(n)} \cong (\cO/\fm^n)^\times$. Hence,
Together with
\begin{align*}
\xymatrix{1 \ar[r] & \displaystyle\frac{U^{(n)}}{U^{(n+r)}} \ar[r] & \displaystyle\lbrb{\frac{\cO}{\fm^{n+r}}}^\times \ar[r] & \displaystyle
\lbrb{\frac{\cO}{\fm^n}}^\times \ar[r] & 1  }
\end{align*}
for $n = 1$ and $r=2$, we obtain
\begin{align*}
\xymatrix{
1 \ar[r] & U_{K_{2}}^{(3)} \ar[r] \ar[d] & U_{K_{2}}^{(1)} \ar[r] \ar[d] &
\ker\lbrb{\displaystyle \lbrb{\frac{\cO_{K_{2}}}{8 \cO_{K_{2}}}}^\times \to \lbrb{\frac{\cO_{K_{2}}}{2\cO_{K_{2}}}}^\times } \ar[r] \ar[d] & 0 \\
1 \ar[r] & U_{F_{2}}^{(3)} \ar[r] & U_{F_{2}}^{(1)} \ar[r] & \ker\lbrb{\displaystyle \lbrb{\frac{\cO_{F_{2}}}{8 \cO_{F_{2}}}}^\times \to \lbrb{\frac{\cO_{F_{2}}}{2\cO_{F_{2}}}}^\times } \ar[r] & 0 
}
\end{align*}
where the vertical arrows are the norm maps from $K_2$ to $F_2$.
Since the norm map on $U^{(3)}$ is surjective, $\Gamma_1/\Gamma_3$ is isomorphic to the kernel of the norm map on the right column.

\begin{proposition} \label{prop:str of X}
The group
\begin{align} \label{eqn:groupX}
\Gamma_1/\Gamma_3 \cong
 \ker \lbrb{\Nm_{K_2/F_2} : \ker\lbrb{ \lbrb{\frac{\cO_{K_{2}}}{8 \cO_{K_{2}}}}^\times \to \lbrb{\frac{\cO_{K_{2}}}{2\cO_{K_{2}}}}^\times } \to \ker\lbrb{ \lbrb{\frac{\cO_{F_{2}}}{8 \cO_{F_{2}}}}^\times \to \lbrb{\frac{\cO_{F_{2}}}{2\cO_{F_{2}}}}^\times }}
\end{align}
is isomorphic to
\begin{align*}
	\bZ/2\bZ \oplus \bZ/2\bZ \oplus (\bZ/4\bZ)^{\oplus \frac{p-3}{2}}.
\end{align*}
Furthermore, the generators are
\begin{itemize}
	\item $(1 + 2)$ which is of order $2$,
	\item $\prod_{i=1}^{\frac{p-1}{2}}(1 + 2\zeta_p^i)(1 + 2\zeta_p^{-i})^{-1}$ which is of order $2$, and
	\item $(1+2\zeta_p^{i})(1+2\zeta_p^{-i})^{-1}$ for $i = 2, \cdots, \frac{p-1}{2}$ each of which is of order $4$.
\end{itemize}
\end{proposition}
\begin{proof}
For the sake of notational simplicity, we denote $H$ for the group (\ref{eqn:groupX}).
%Since $2$ is inert in $K/\bQ$, each element of $\cO_{K_{2}}/8\cO_{K_{2}}$ can be written by
%\begin{align*}
%	a_0 + a_1 \zeta_p + \cdots a_{p-2}\zeta_p^{p-2}, \qquad a_i \in \bZ/8\bZ
%\end{align*}
%and element of $\cO_{F_{2}}/8\cO_{F_{2}}$ written by
%\begin{align*}
%	b_0 + b_1 (\zeta_p+\zeta_p^{-1}) + \cdots + b_{\frac{p-1}{2} - 1} (\zeta_p^{\frac{p-1}{2}-1} + \zeta_p^{-\frac{p-1}{2}+1}), \qquad b_i \in \bZ/8\bZ.
%\end{align*}
By Theorem \ref{thm:Galoisunit}, $(\cO_{K_{2}}/8\cO_{K_{2}})^\times$ is isomorphic to a product of a group of roots of unity and $\bZ/2\bZ \oplus \bZ/2\bZ \oplus (\bZ/4\bZ)^{\oplus (p-2)}$.
Since $2$ is inert in $K/\bQ$, $(\cO_{K_2}/2\cO_{K_2})^\times$ is the group of roots of unity.
Hence
\begin{align*}
	H_K := \ker\lbrb{ \lbrb{\frac{\cO_{K_{2}}}{8 \cO_{K_{2}}}}^\times \to \lbrb{\frac{\cO_{K_{2}}}{2\cO_{K_{2}}}}^\times }
	\cong \bZ/2\bZ \oplus \bZ/2\bZ \oplus (\bZ/4\bZ)^{\oplus (p-2)}
\end{align*}
which is generated by
\begin{align*}
	\lcrc{1+2, 1+4\zeta_p} \cup \lcrc{1 + 2\zeta_p^i}_{i=1}^{p-2}.
\end{align*}
Similarly,
\begin{align*}
	H_F := \ker\lbrb{ \lbrb{\frac{\cO_{F_{2}}}{8 \cO_{F_{2}}}}^\times \to \lbrb{\frac{\cO_{F_{2}}}{2\cO_{F_{2}}}}^\times }
	\cong \bZ/2\bZ \oplus \bZ/2\bZ \oplus (\bZ/4\bZ)^{\oplus \frac{p-3}{2}}
\end{align*}
and this is generated by
\begin{align*}
	\lcrc{3, 1+4(\zeta_p + \zeta_p^{-1})} \cup \lcrc{1 + 2(\zeta_p + \zeta_p^{-1})^i}_{i=1}^{\frac{p-3}{2}}.
\end{align*}

One can show that $\Nm_{K_2/F_2}(1 + 2\zeta_p^i) = 5 + 4(\zeta_p^i + \zeta_p^{-i})$ is an element of 
$H_F$ of order $4$.
Also, $\Nm_{K_2/F_2}(1 + 2\zeta_p^i)$ for $i = 2,  \cdots, \frac{p-1}{2}$ give $\frac{p-3}{2}$ independent elements of order $4$ in $H_F$.
On the other hand, for an odd prime $p$, we have
\begin{align*}
\Nm_{K_2/F_2}(1 + 4\zeta_p) & \equiv 1 + 4(\zeta_p + \zeta_p^{-1}) \pmod{8\cO_{F_{2}}}
\end{align*}
and
\begin{align*}
	\prod_{i=1}^{\frac{p-1}{2}}(5 + 2(\zeta_p^i + \zeta_p^{-i})) \equiv 3 \pmod{8}.
\end{align*}
Hence the norm map from $H_K$  to $H_F$ is surjective. Consequently, we have
\begin{align*}
	H = \ker\lbrb{\Nm_{K_2/F_2} : H_K \to H_F} \cong \bZ/2\bZ \oplus \bZ/2\bZ \oplus (\bZ/4\bZ)^{\oplus \frac{p-3}{2}}.
\end{align*}

The norm computation shows that $3$ and $(1+2\zeta_p^{i})(1+2\zeta_p^{-i})^{-1}$ for $i = 2, \cdots, p-2$  in $H_K$ are in $H$ and the order is $2$ and $4$ respectively.
Consider an element
\begin{align*}
	h = \frac{1 + 2\zeta_p}{1+2\zeta_p^{-1}} \cdot \frac{1 + 2\zeta_p^2}{1+2\zeta_p^{-2}} \cdots
	\frac{1 + 2\zeta_p^{\frac{p-1}{2}}}{1+2\zeta_p^{\frac{p+1}{2}}} \in H.
\end{align*}
Since $(1+2)(1 + 2\zeta_p) \cdots (1 + 2\zeta_p^{p-1}) = 1 + 2^{p}$, we have
\begin{align*}
	(1+2\zeta_p)(1 + 2\zeta_p^{-1})^{-1} \equiv (1+2)(1 + 2\zeta_p)^2 (1 + 2\zeta_p^2) \cdots (1 + 2\zeta_p^{p-2}) \pmod{8\cO_{K_2}}.
\end{align*}
Therefore,
\begin{align*}
	\frac{1 + 2\zeta_p}{1+2\zeta_p^{-1}} \cdot \frac{1 + 2\zeta_p^2}{1+2\zeta_p^{-2}} \cdots
	\frac{1 + 2\zeta_p^{\frac{p-1}{2}}}{1+2\zeta_p^{\frac{p+1}{2}}} \equiv
	(1 + 2)(1 + 2\zeta_p)^2 (1 + 2\zeta_p^2)^2 \cdots (1 + 2\zeta_p^{\frac{p-1}{2}})^2 \pmod{8\cO_{K_2}}.
\end{align*}
Since $(1 + 2\zeta_p^i)^4 \equiv 1 \pmod{8\cO_{K_{2}}}$,
\begin{align*}
	h^2 &\equiv (1 + 2)^2(1 + 2\zeta_p)^4 (1 + 2\zeta_p^2)^4 \cdots (1 + 2\zeta_p^{\frac{p-1}{2}})^4\equiv1 \pmod{8\cO_{K_2}}.
\end{align*}
Hence we found an element of order $2$ in $H$, which is not in a subgroup of $H$ generated by $(1+2), (1+2\zeta_p^i)(1+2\zeta_p^{-i})^{-1}$ for $i = 2, \cdots, p-2$.
\end{proof}

\begin{lemma} \label{lem:repofga}
	The group $\Gamma_1/\Gamma_3$ is generated by
\begin{align*}
\left\{\frac{\zeta_p^{\frac{p-1}{2}} + 3\zeta_p^{-\frac{p-1}{2}}}{3\zeta_p^{\frac{p-1}{2}} + \zeta_p^{-\frac{p-1}{2}}}\right\}\cup\left\{\prod_{i=1}^{\frac{p-1}{2}}\frac{1 + 2\zeta_p^i}{1 + 2\zeta_p^{-i}}\right\}\cup\left\{\frac{1+2\zeta_p^{i}}{1+2\zeta_p^{-i}}\right\}_{i=2,\cdots,\frac{p-1}{2}}
\end{align*}
which are of norm $1$. Moreover, along the map $\mathbb{P}^1(F_2)\rightarrow G_2$ sending $a$ to $g_a$,
	\begin{itemize}
	\item $2(\zeta_p^{\frac{p-1}{2}} + \zeta_p^{-\frac{p-1}{2}})$ corresponds to $\displaystyle\frac{\zeta_p^{\frac{p-1}{2}} + 3\zeta_p^{-\frac{p-1}{2}}}{3\zeta_p^{\frac{p-1}{2}} + \zeta_p^{-\frac{p-1}{2}}}$,
	\item $\displaystyle\frac{\prod_{i=1}^{\frac{p-1}{2}}(1+2\zeta_p^i) + \prod_{i=1}^{\frac{p-1}{2}}(1 + 2\zeta_p^{-i})    }{\prod_{i=1}^{\frac{p-1}{2}}(1+2\zeta_p^i) - \prod_{i=1}^{\frac{p-1}{2}}(1 + 2\zeta_p^{-i})}( \zeta_p^{-\frac{p-1}{2}}-\zeta_p^{\frac{p-1}{2}})$ corresponds to $\displaystyle\prod_{i=1}^{\frac{p-1}{2}}\frac{1 + 2\zeta_p^i}{1 + 2\zeta_p^{-i}}$,
	\item $\displaystyle\frac{(1 + \zeta_p^{i} + \zeta_p^{-i})}{\zeta_p^i - \zeta_p^{-i}}( \zeta_p^{-\frac{p-1}{2}}-\zeta_p^{\frac{p-1}{2}})$ corresponds to $\displaystyle\frac{1+2\zeta_p^{i}}{1+2\zeta_p^{-i}}$ for each $i=2,\cdots,\frac{p-1}{2}$.
\end{itemize}	
\end{lemma}
\begin{proof}
By Lemma \ref{lem:SYlem5.1} and Proposition \ref{prop:str of X}, we need to find an $a \in \cO_{F_{2}}$ satisfying
\begin{align*}
	g_a \equiv 1 + 2 \pmod{8\cO_{K_{2}}}, \qquad a^2 \not \equiv \Delta \pmod{4\cO_{F_{2}}},
\end{align*} 
%The first condition is equivalent to
%\begin{align*}
%	2a \equiv 4\delta \pmod{8\cO_{K_{2}}}
%\end{align*}
%where $\delta = \zeta_p^{-\frac{p-1}{2}}-\zeta_p^{\frac{p-1}{2}}.$
One can check that  $a = 2(\zeta_p^{\frac{p-1}{2}} + \zeta_p^{-\frac{p-1}{2}})$ and corresponding $g_a$ satisfy the conditions. On the other hand, $\prod_{i=1}^{\frac{p-1}{2}}(1 + 2\zeta_p^i)(1 + 2\zeta_p^{-i})^{-1}$ and $(1+2\zeta_p^{i})(1+2\zeta_p^{-i})^{-1}$ for $i = 2, \cdots, \frac{p-1}{2}$ are in the kernel of the norm. Using the relation between $g_a$ and $a$, one can easily compute $a$.
\end{proof}

From now on, we concretely study the case of $K = \bQ(\zeta_{11})$.

\begin{lemma} \label{lem:11Gamma1}
	Let $p = 11$, let $\beta$ be as in Lemma \ref{lem:Weilunity}, and let $u=2\beta$.
	Assume that $\alpha \in \bZ_2^\times$. Then, $\gamma_F(\psi) = \frac{-1+i}{\sqrt{2}}$ and
	\begin{align*}
	\begin{array}{|c|c|c|c|c|c|c|c|c|c|c|}
	\hline
		g_a 						& 1+2	& \frac{1+2\zeta_p^{2}}{1+2\zeta_p^{-2}} &\frac{1+2\zeta_p^{3}}{1+2\zeta_p^{-3}} &\frac{1+2\zeta_p^{4}}{1+2\zeta_p^{-4}} & \frac{1+2\zeta_p^{5}}{1+2\zeta_p^{-5}} & \prod_{i=1}^{\frac{p-1}{2}} \frac{1 + 2\zeta_p^i}{1 + 2\zeta_p^{-i}} 	\\ \hline
	\gamma_{F,2}(a^2+\Delta,\psi)	& 1		&  -1 & -1 & -1 & -1 & 1 \\ \hline
	(a\Delta, a^2+\Delta)_{F, 2} & 1 & 1 & -1 & -1 & -1 & -1 \\ \hline
	\Tr_{F_2/\bQ_2}\lbrb{\frac{a\alpha u^2 \Delta^2}{a+\Delta}} \pmod{4} & 2 & 2 & 0 & 0 & 0 & 0
	\\ \hline
	\end{array}
	\end{align*}
\end{lemma}
\begin{proof}
    For the numerical computations in the proof, see \texttt{Weilindex11} which is a Sagemath notebook \cite{Sag}.
	Denote $\zeta = \zeta_{11}$ and $\omega = \zeta + \zeta^{-1}$ in the proof.
Following Lemma \ref{lem:repofga},
we define
\begin{align*}
    a_1 &:= 2(\zeta^5 + \zeta^6) = -2\omega^4 - 2\omega^3 + 6\omega^2 + 4\omega - 2, \\
    a_2 &:= \frac{(1+\zeta^2+\zeta^{-2})(\zeta^6-\zeta^5)}{\zeta^2-\zeta^{-2}}
    = -\omega^3 + \omega^2 + 2\omega - 2, \\
    a_3 &:= \frac{(1+\zeta^3+\zeta^{-3})(\zeta^6-\zeta^5)}{\zeta^3-\zeta^{-3}}
    = -\omega^4 + 3\omega^2 - \omega + 1 ,\\
    a_4 &:= \frac{(1+\zeta^4+\zeta^{-4})(\zeta^6-\zeta^5)}{\zeta^4-\zeta^{-4}}
    =-\omega^3 + \omega^2 + 3\omega - 3, \\
    a_5 &:= \frac{(1+\zeta^5+\zeta^{-5})(\zeta^6-\zeta^5)}{\zeta^5-\zeta^{-5}}
    = \omega^4 + \omega^3 - 3\omega^2 - 2\omega.
\end{align*}
Then it maps to $g_{a_i}$ in the previous lemma.
Also, we have
	\begin{align*}
	&\lbrb{	\frac{\prod_{i=1}^5\lbrb{\frac{1+2\zeta^i}{1-2\zeta^{-i}}} + 1}{\prod_{i=1}^5\lbrb{\frac{1+2\zeta^i}{1-2\zeta^{-i}}} -1} } (\zeta^6 - \zeta^5)	
	= \frac{-1}{3917}\lbrb{251598\omega^4 + 184454\omega^3 - 1058498\omega^2 - 468084\omega + 893580}
	\end{align*}
which is denoted by $a_6$.
We note that $\Delta = (\zeta^6 - \zeta^5)^2 = \omega - 2$ and
\begin{align*}
    u = 2\beta=\left(\alpha(\zeta_{p}^{\frac{p-1}{2}}+\zeta_{p}^{-\frac{p-1}{2}})^{3}\right)^{2^{k_v-1}-1}
    = \left(\alpha(\zeta^{5}+\zeta^{-5})^{3}\right)^{2^{4}-1}
    = \alpha^{15}(\zeta^5 + \zeta^6)^{45}.
\end{align*}

Since $\alpha + 2\bZ_2$ gives same $\Tr_{F_2/\bQ_2}\lbrb{\frac{a\alpha u^2 \Delta^2}{a+\Delta}} \pmod{4}$, we can assume that $\alpha = 1$. 
Then, we can evaluate the traces by a numerical computation.
The computation of the Hilbert symbol also can be found in the Sagemath notebook.

On the other hand, we can check that $a_i^2 + \Delta$ is relatively prime to $2$ in $F$.
Hence, $m((a^2 + \Delta)\psi) = m(\psi)$.
It ensures the validity of the function in the Sagemath notebook.
For $a_6$, we use $\gamma_{F, 2}(3917^2(a_6^2 + \Delta),\psi)$, instead of $\gamma_{F, 2}(a_6^2 + \Delta, \psi)$.
\end{proof}

\begin{proof}[Proof of Theorem \ref{main:on11}]
Denote $\zeta := \zeta_{11}$ in the proof.
By (\ref{eqn:xi rep}), the Weil index condition (\ref{eqn:Weilcondat1}) is equivalent to 
\begin{align*}
    i^T (a\Delta,a^{2}+\Delta)_{F,2}\gamma_{F,2}(a^{2}+\Delta,\psi) = \chi_{2}(-\delta y(g-1))
\end{align*}
where $T = \Tr_{F_v/\bQ_2}\lbrb{\frac{a \alpha u^2 \Delta^2}{a^2 + \Delta}}$, since $(a\Delta,a^{2}+\Delta)_{F,2}\gamma_{F,2}(a^{2}+\Delta,\psi) \in \lcrc{\pm 1}$.
By the assumption, $\alpha$ is odd and 
\begin{align*}
	\chi_2(-\delta y (g-1)) = (-1)^{\ord_2(g-1)+\ord_2(y)}.
\end{align*}
Therefore for the generators $g_a$ listed in Proposition \ref{prop:str of X}, 
the Weil index condition (\ref{eqn:Weilcondat1}) is equivalent to 
\begin{align*}
\chi_2(-\delta y (g-1)) = \left\{
\begin{array}{llll}
-1 & \textrm{when } g = 1+2 \\
+1 & \textrm{when } g = \frac{1+2\zeta^{2}}{1+2\zeta^{-2}} \\
+1 & \textrm{when } g = \frac{1+2\zeta^{3}}{1+2\zeta^{-3}} 
\end{array}
\right.
\qquad
\chi_2(-\delta y (g-1)) = \left\{
\begin{array}{llll}
+1 & \textrm{when } g = \frac{1+2\zeta^{4}}{1+2\zeta^{-4}} \\
+1 & \textrm{when } g = \frac{1+2\zeta^{5}}{1+2\zeta^{-5}} \\
-1 & \textrm{when } g=  \prod_{i=1}^{5} \frac{1 + 2\zeta^i}{1 + 2\zeta^{-i}}
\end{array}
\right.
\end{align*}
by Lemma \ref{lem:11Gamma1}. 
Since $2$ is inert in $F/\bQ$, the valuation of an element in $K$ at $2$ agrees with that of in $F$. Hence, 
\begin{align*}
	\ord_2\lbrb{\frac{1+2\zeta^i}{1-2\zeta^{-i}} - 1}
	= \ord_2(1+2\zeta^i -(1-2\zeta^{-i})) = \ord_2(2(\zeta^i + \zeta^{-i})) = 1.
\end{align*}
By (\ref{eqn:xy for ga}), we have $y = \frac{2a}{a^2-\Delta}$ if $g = g_a = x + y \delta$.
So, we can compute
\begin{align*}
	\ord_2(y) = \left\{
	\begin{array}{lll}
	2 & g = 1+2,\,\,  \prod_{i=1}^{5} \frac{1 + 2\zeta^i}{1 + 2\zeta^{-i}}, \\
	1 & \textrm{otherwise.}
	\end{array}
	\right.
\end{align*}
One can check the calculation in \texttt{Weilindex11}. %_\emph{index}_\emph{at}_\emph{11}.
Hence we have verified (\ref{eqn:Weilcondat1}).
The Weil index condition at roots of unity can be checked by Remark \ref{rmk:Weilcond roots}.
% % For the condition of the Weil index at roots of unity  (\ref{eqn:Weilcondatzeta}), we note that $\ord_2(y_j) = -1$. Hence
% % $\chi_{2}(-\delta y_{j}(\zeta^{j}-1)) = -1$ and the condition is equivalent to 
% % \begin{align*}
% % \sum_{j=1}^{10}\gamma_{F, {2}}(\alpha y_{j}(1-x_{j})\psi)\psi\left(\frac{\Delta\alpha(x_{j}-1)}{y_{j}}\beta^{2}\right) \neq \pm 1.
% % \end{align*}
% It is directly followed by Lemma \ref{lem:Weilcondatroots11}.
\end{proof}

\section{Fermat curves and Hyperelliptic curves} \label{sec:Fermat}
\subsection{Hecke character attached to Fermat curves}
For a prime ideal $\fl$ of $\cO_K$ over a rational prime $\ell \nmid p$, 
we define $\lbrb{\frac{a}{\fl}}_p$ to be the $p$-th root of unity which is uniquely determined by the condition
\begin{align*}
    \lbrb{\frac{a}{\fl}}_p \equiv a^{\frac{\Nm_{K/\bQ}\fl-1}{p}} \pmod{\fl}.
\end{align*}
We also define a character
\begin{align*}
    \chi_{\fl}(a) := \lbrb{\frac{a}{\fl}}_p.
\end{align*}
For an element $d \in \bQ^\times/\bQ^{\times p}$, let $\chi_d$ be the character supported on the prime ideals of $\cO_K$ relatively prime to $dp$, defined to be
\begin{align*}
    \chi_d(\fl) = \pthsym{d}{\fl}.
\end{align*}
Let $\cF(p)$ be the $p$-th Fermat curve defined by an equation $x^p + y^p = 1$ and let $C_{r, s, t}$ be the curve defined by an equation $y^p = x^r(1-x)^s$, where $r, s, t$ are positive integers satisfying $r+s+t =p$.
Then the Jacobian of $\mathcal{F}(p)$ is isogeneous with the product of the Jacobian of $C_{1, s, p-s-1}$ for $s = 1, \cdots, p-2$.
Let $J_{r, s, t}$ be the Jacobian of $C_{r, s, t}$.
Then $J_{r, s, t}$ has complex multiplication and the associated Hecke character $\chi_{r, s, t}$ is determined by
\begin{align*}
	\chi_{r, s, t}(\fl) = - \frac{J(\chi_{\fl}^r, \chi_{\fl}^s)}{| J(\chi_{\fl}^r, \chi_{\fl}^s) |}
\end{align*}
whenever $\fl \nmid p$ (cf. \cite[\S 3]{GR78}, \cite[\S 2]{Wei52}).
Here $J(\chi_1, \chi_2)$ denotes the Jacobi sum 
\begin{align*}
    J(\chi_1, \chi_2) = - \sum_{\substack{x \in \cO_K/\fl \\ x \neq 0, 1}}\chi_1(x)\chi_2(1-x).
\end{align*}
For the reader's convenience, we note that our $\chi_{r,s,t}$ is $\chi_{r, s, t}'$ in \cite{GR78}.
%Let $\lara{n}$ be the unique integer in $[0,p)$ equivalent to $n$ modulo $p$ and 
%\begin{align*}
%    H_{r, s, t} := \lcrc{h \in (\bZ/p\bZ)^\times : 
%   \lara{hr} + \lara{hs} + \lara{ht} = p}.
%\end{align*}
%
%\begin{align*}
%    m = \sum_{h \in H_{r, s, t}} \lara{h^{-1}}.
%\end{align*}
%Then, $H_{r, s, t}$ may be identified with CM-type of $J_{r, s, t}$.
%For each $h \in H_{r, s, t},$ we denote $\sigma_h$ by the automorphism of $K$ over $\bQ$ which satisfies $\sigma_h(\zeta_p) = \zeta_p^h$.
For the basic properties of $\chi_{r, s, t}$, we refer \cite[\S 3]{GR78}.
Weil (cf. \cite[\S2]{Wei52}, \cite[\S 2.1]{Shu}) showed that
\begin{align*}
    \chi_{r, s, t}^{(d)} := \chi_d^{r+s}\chi_{r, s, t}
\end{align*}
is the Hecke character associated to $J^{(d)}_{r, s, t}$, the Jacobian of the curve $y^p = x^r(d-x)^s.$
%Note that the Jacobian of $x^p + y^p = d$ is isogenuous to some product of $J_{r, s, t}^{(d)}$ (cf. Introduction).
To state the properties of $\chi_{r, s, t}^{(d)}$, we introduce some notations.
Along the usual isomorphism $\mathbb{Q}_p^\times\cong p^\mathbb{Z}\times\mu_{p-1}\times(1+p\mathbb{Z}_p)$, where $\mu_{p-1}$ is the group of $(p-1)$-th roots of unity, each $x\in\mathbb{Q}_p^\times$ admits a unique decomposition:
\begin{align*}
    x=\epsilon p^a (1-p)^b,\quad\epsilon\in\mu_{p-1},\quad a\in\mathbb{Z},\quad b\in\mathbb{Z}_p
\end{align*}
Note that $(1-p)^b$ converges (cf. \cite[p.81-82]{Kob}) and $b$ is determined by taking the $p$-adic logarithm of $\epsilon^{-1}p^{-a}x$ (cf. \cite[p.78-81]{Kob}). 
We denote
\begin{align*}
    u(x) := \min(\ord_p(a), \ord_p(b)+1).
\end{align*}

\begin{lemma} \label{lem:chi rst d}
(i) The restriction of $\chi^{(d)}_{r, s, t}$ on $\bA_F^\times$ is the quadratic character $\epsilon_{K/F}$ associated to $K/F$.\\
(ii) If $w$ is infinite, then $\chi_{r, s, t;w}^{(d)}(z) = |z|/z$ and the root number is $i^{-1}$.\\
(iii) If $w$ does not divide $pd$, then $\chi_{r, s, t}^{(d)}$ is unramified at $w$ and
\begin{align*}
    \chi_{r, s, t}^{(d)}(\fl) = \pthsym{d}{\fl}^{r+s} 
    \frac{J(\chi_\fl^r, \chi_\fl^s)}
    {|J(\chi_\fl^r, \chi_\fl^s) |}
\end{align*}
for prime ideals $\fl \nmid p$. 
The conductor exponent of $\chi^{(d)}_{r, s, t}$ at $w$ is zero.%, and the root number is $+1$.
\\
(iv) If $w$ divides $p$ and $u(r^rs^s(t-p)^td^{r+s}) \geq 2$, then the conductor exponent of $\chi^{(d)}_{r, s, t}$ at $w$ is $1$.% and the root number is $\quadsym{2}{p}i^{\frac{p-1}{2}}$.
\end{lemma}
\begin{proof}
By \cite[Proposition 5]{Roh92}, we know that the Jacobi sum Hecke characters $\chi_{r, s, t}$ agree with $\epsilon_{K/F}$ and by \cite[Lemma 3.2]{Shu}, $\chi_d$ is trivial on $\bA_{F}^\times$ so (i) follows. The other statements come from \cite[Lemma 3.10, Proposition 3.11, Proposition 3.12]{Shu}.
\end{proof}

\begin{corollary} \label{cor:Fermatalpha}
For the Hecke character $\chi_{r, s, t}^{(d)}$, we may choose $\alpha$ to be positive.
\end{corollary}
\begin{proof}
We note that the local root number is $\epsilon\lbrb{\frac{1}{2}, \chi_w, \frac{1}{2}\psi_w }$  in \cite{SY03} but is $\epsilon\lbrb{\frac{1}{2}, \chi_w, \psi_w }$ in \cite{GR78, Shu}.
The difference can be measured by the formula
\begin{align*}
	\epsilon\left(\frac{1}{2}, \chi_w, \frac{1}{2}\psi_{w}\right) = \frac{\chi_w(1/2)}{|\chi_w(1/2)|}\epsilon\left(\frac{1}{2}, \chi_w, \psi_{w}\right).
\end{align*}
By Lemma \ref{lem:chi rst d} (ii),
\begin{align*}
\frac{\chi_{r, s, t;w}^{(d)}(\frac{1}{2})}{|\chi_{r, s, t;w}^{(d)}(\frac{1}{2})|}=1, \qquad
\epsilon\left(\frac{1}{2}, \chi_{r, s, t;w}^{(d)}, \psi_w \right) = i^{-1}, \qquad
\chi_{r, s, t;w}^{(d)}(\delta) = i
\end{align*}
at each infinite place. Hence, the left hand side of 
\begin{align*}
\prod_{w \mid v} \epsilon\left(\frac{1}{2}, \chi_{r, s, t;w}^{(d)}, \frac{1}{2}\psi_w \right) \chi_{r, s, t;w}^{(d)}( \delta )  = \epsilon_{K/F; v}(\alpha)
\end{align*}
is $+1$, and it means that $\sigma_v(\alpha)$ is positive.
\end{proof}

\begin{lemma} \label{lem:globalShu}
	Let $d$ be a squarefree product of a rational prime $\ell \neq p$ such that
 \begin{align*}
     u(r^rs^s(t-p)^td^{r+s}) \geq 2.
 \end{align*}
 Then, the global root number of $\chi_{r, s, t}^{(d)}$ is
	\begin{align*}
		\quadsym{2}{p}\cdot \prod_{\ell \mid d} \quadsym{\ell}{p}.
	\end{align*}
\end{lemma}
\begin{proof}
\cite[Theorem 3.13]{Shu}.
\end{proof}

\begin{theorem} \label{thm:mainFermat}
	Let $p$ be a prime satisfying $\quadsym{2}{p} = +1$ and let $r, s, t$ be positive integers satisfying $r+s+t=p$.
	Then, there are infinitely many squarefree integers $d$ satisfying $L\lbrb{1, \chi_{r, s, t}^{(d)}} \neq 0$.
\end{theorem}
\begin{proof}
We first note that the restriction of $\chi_{r, s, t}^{(d)}$ to $\bA_F^\times$ is a quadratic character by Lemma \ref{lem:chi rst d} (i),
and the infinite part of $\chi_{r, s, t}^{(d)}$ is $|z|/z$ by Lemma \ref{lem:chi rst d} (ii).
Hence the character satisfies the infinite place condition of Theorem \ref{main:main}.
Let $d$ be a squarefree integer such that 
	\begin{enumerate}
		\item $2 \nmid d$,
		\item $u(r^rs^s(t-p)^td^{r+s}) \geq 2$, and
		\item if $\ell \mid d$, then all primes of $F$ dividing $\ell$ split in $K/F$.
	\end{enumerate}
	We note that (3) implies the split condition of Theorem \ref{main:main}, and (2) implies $p$-condition of Theorem \ref{main:main} by Lemma \ref{lem:chi rst d} (iv).
	By Lemma \ref{lem:globalShu} and splitting condition, the root number of $\chi_{r, s, t}^{(d)}$ is $+1$. Hence $L$-value does not vanish by Theorem \ref{main:main}.
	
	It suffices to show that there are infinitely many $d$ satisfying the conditions (1)-(3). For each $r, s, t$ such that $1 \leq r, s, t \leq p-2$ and $r + s + t = p$, there exists $x \in \bZ/p^2\bZ$ satisfying $r^rs^s(t-p)^t \equiv x \pmod{p^2}$. 
	Since $r, s, (t-p)$ are relatively prime to $p$, $x$ is a unit in $\bZ/p^2\bZ$. Therefore, there exists $d_0 \in (\bZ/p^2\bZ)^\times$	such that
	\begin{align} \label{eqn:d0condi}
		(d_0^{(r+s)(p-1)})^{-1} \equiv x^{p-1}
	\end{align}
	and $d_0$ is determined by the value at $\bZ/p\bZ$-copy in $(\bZ/p^2\bZ)^\times \cong \bZ/p\bZ \oplus \bZ/(p-1)\bZ$.
 Also, there is a unique $d_0$ which is equivalent to a given element of $(\bZ/p\bZ)^\times \cong \bZ/(p-1)\bZ$ modulo $p$ and satisfying (\ref{eqn:d0condi}).	
	Choose $d_0 \in (\bZ/p^2\bZ)^\times$ which is a quadratic residue modulo $p$ and satisfying (\ref{eqn:d0condi}). 
	Let $d$ be a prime which is equivalent to $d_0$ modulo $p^2$.
	Then it splits in $K/F$ since it is quadratic residue modulo $p$, and $(r^rs^s(t-p)^td^{r+s})^{p-1} \equiv 1 \pmod{p^2}$.
	There are infinitely many such $d$.\footnote{It suffices to prove this theorem, but it is easy to construct composite $d$, by taking the product of such prime and primes equivalent to $1$ modulo $p^2$.}
		
Since
	\begin{align*}
		u(x) = \min(\ord_p(a), \ord_p(b)+1)
	\end{align*}
	where $x = \epsilon p^a(1-p)^b$ for $\epsilon \in \mu_{p-1}$, we have
	\begin{align*}
		u(r^r s^s (t-p)^t d^{r+s}) &= \ord_p\lbrb{\frac{(r^r s^s (t-p)^t d^{r+s})^{p-1}-1}{p}} +1 \geq 2.
	\end{align*}
	Hence such $d$ satisfies (1)-(3).
\end{proof}

\begin{proof}[Proof of Theorem \ref{main:Fermat}]
It is a direct consequence of Theorem \ref{thm:mainFermat}.
\end{proof}

The proof of Theorem \ref{thm:mainFermat} also gives the following simultaneous nonvanishing theorem.

\begin{corollary} \label{cor:simul}
	Let $p$ be a prime satisfying $\quadsym{2}{p} = +1$. Then, there are infinitely many $d$ such that
	$L(1, \chi_{r, s, t}^{(d)}) \neq 0$ for all $r, s, t$ satisfying $r+s+t = p$ and 
	$(r^rs^s(t-p)^t)^{p-1} \equiv a \pmod{p^2}$ for an $a \in (\bZ/p^2\bZ)^\times$.
\end{corollary}

\begin{example} \label{exam:simul}
If $p = 31$ and $(r, s, t)$ is one of
\begin{align*}
	(1, 5, 25),\, (2, 10, 19), \, (3, 13, 15), \, (4, 7, 20), \, (8, 9, 14),
\end{align*}
then $(r^r s^s (t-p)^t)^{p-1} \equiv 1 \pmod{p^2}$. Hence for a prime $d \equiv 1 \pmod{p^2}$, 
the special values of the five $L$-functions of $\chi_{r, s, t}^{(d)}$ are simultaneously nonvanishing.
\end{example}

\subsection{Hecke characters attached to some hyperelliptic curves}

In this section, we will show that the Hecke characters $\chi_{A}$ over $\bQ(\zeta_{11})$ attached to the hyperelliptic curves
\begin{align*}
	y^2= x^{11} + A^2
\end{align*}
satisfy Condition \ref{cond:Weil}.
Note that $2$ is inert in $\bQ(\zeta_{11})/\bQ$ and 
the description of $\chi_A$ is concretely given in \cite{Wei52, Sto02}.
%Furthermore, we can show the following.

\begin{lemma} \label{lem:11hyperelliptic}
Assume that $A$ is the $11$-th power free integer such that $(A, 22) = 1$, $11^2 \mid (A^{10}-1)$  and all primes dividing $A$ split in $K/F$.
\\
(i) The restriction of $\chi_A$ to $F$ is the quadratic Hecke character $\epsilon_{K/F}$.\\
(ii) If $w$ is infinite, then $\chi_A(z) = |z|/z$ and  the root number of $\chi_A$ is $i^{-1}$.\footnote{We note that \cite[Proposition 2.2 (i)]{SY03} includes small typo: the root number should be $-i$ and $\chi_{A, w}(\delta) = i$. }\\
(iii) If $w = 2$, then the conductor exponent of $\chi_{A}$ at $2$ is $1$ if $A \equiv 1 \pmod{4}$ and $2$ if $A \equiv 3 \pmod{4}$. We also have $\chi_{A, 2}(2) = -1 $ and 
\begin{align*}
	\chi_{A, w}(x) = \quadsym{-1}{\Nm_{K_w/F_v} (x)}^{\frac{A-1}{2}} \lbrb{\frac{x}{2}}_{11}^{2}
\end{align*}
on  $\cO_{K_w}^\times$. The root number is $(-1)^{1+f_2}$ where $f_2$ is the conductor exponent of $\chi_{A, 2}$.
\\
(iv) If $w$ is the unique prime above $11$, then the conductor exponent of $\chi_A$ at $w$ is $1$.\\
%On $\cO_{K, w}^\times$, we have
%\begin{align*}
%	\chi_{A, w}(x) = \lbrb{\frac{x}{\pi_w}}_2.
%\end{align*}
%The root number of $\chi_A$ at $w$ is $-1$.\\
(v) The global root number of $\chi_A$ is $(-1)^{1+f_2}$.
\end{lemma}
\begin{proof}
	In the proof of \cite[Proposition 3.3]{Sto02}, there is an exact form of the finite part of $\chi_A$, which proves (i) as well. Also, \cite[Theorem 3.2]{Sto02} gives a formula for global root numbers.
	\cite[Lemma 2.1, Proposition 2.2]{SY03} gives the other statements.
\end{proof}

\begin{corollary} \label{cor:Hyperalpha}
Let $A$ be an integer equivalent to $1$ modulo $4$. 
For the Hecke character $\chi_A$, we may choose $\alpha$ to be odd and positive.
\end{corollary}
\begin{proof}
The proof of Corollary \ref{cor:Fermatalpha} also shows that $\alpha$ is positive in this case.
On the other hand,
\begin{align*}
	 \epsilon\left(\frac{1}{2}, \chi_{A, 2}, \frac{1}{2} \psi_{2}\right)
	 \chi_{A, 2}(\delta) = (-1)^{1 + f_2} (-1)^{\frac{A-1}{2}} = 1.
\end{align*}
Hence $\alpha$ should satisfy $\ord_2(\alpha) \equiv 0 \pmod{2}$.
\end{proof}

\begin{proof}[Proof of Theorem \ref{main:hyperelliptic}]
By Lemma \ref{lem:11hyperelliptic}, $\chi_{A, 2}(-\delta) = (-1)^{\frac{A-1}{2}}$ and $\chi_{A, 2}(2) = -1$. Since we further assume that $A \equiv 1 \pmod{4}$, $\chi_{A, 2}(-\delta) = +1$.
Together with $\alpha \in \bZ_2^\times$, the proof of Theorem \ref{main:on11} also works in this case.
Therefore, the Hecke characters attached to the hyperelliptic curves satisfy the Weil index condition.
Since $11^2 \mid (A^{10}-1)$ (resp. all primes dividing $A$ splits in $K/F$), it also satisfies $p$-condition (resp. splitting condition) of Theorem \ref{main:main}.
Also Corollary \ref{cor:Hyperalpha} shows that $\chi_A$ satisfies the infinite place condition of Theorem \ref{main:main}.
By Theorem \ref{main:main}, its $L$-function does not vanish at $s=1$.
\end{proof}

\emph{Acknowledgement.} K. Jeong was supported by Chonnam National University (Grant number: 2022-2642). Y-W. Kwon was supported by the National Research Foundation of Korea (NRF) grant funded by the Korean government (MSIT) (2020R1A4A101664913) and by Basic Science Research Program through the National Research Foundation of Korea (NRF) funded by the Ministry of Education (2022R1I1A1A01067581). J. Park was supported by Samsung Science and Technology Foundation under Project Number SSTF-BA2001-02.

\end{document}